\documentclass[12pt]{amsart}
\usepackage{amssymb,amsmath,amscd,enumerate,verbatim,fullpage}
\usepackage[all]{xy}
\usepackage{enumerate}  

\numberwithin{equation}{section}

\newtheorem{thm}{\bf Theorem}[section]
\newtheorem{lem}[thm]{\bf Lemma}

\newtheorem{cor}[thm]{\bf Corollary}
\newtheorem{prop}[thm]{\bf Proposition}

\theoremstyle{definition}
\newtheorem{ex}[thm]{Example}
\newtheorem{rem}[thm]{Remark}
\newtheorem{quest}[thm]{Question}
\newtheorem*{thm*}{Theorem}

\DeclareMathOperator{\depth}{depth}

\DeclareMathOperator{\height}{ht}

\DeclareMathOperator{\pd}{pd}

\DeclareMathOperator{\grade}{grade}

\DeclareMathOperator{\conv}{conv}

\newcommand{\ZZ}{{\mathbb Z}}
\newcommand{\NN}{{\mathbb N}}

\newcommand{\RR}{{\mathbb R}}

\def\D{{\Delta}}

\def\F{{\mathcal F}}

\def\mm{{\mathfrak m}}
\def\nn{{\mathfrak n}}

\def\a{{\mathbf a}}
\def\b{{\mathbf b}}
\def\c{{\mathbf c}}
\def\e{{\mathbf e}}

\def\v{{\mathbf v}}

\begin{document}

\title{Depth functions of symbolic powers\\ of homogeneous ideals}

\author{Hop Dang Nguyen}
\address{Institute of Mathematics \\ Vietnam Academy of Science and Technology, 18 Hoang Quoc Viet \\ Hanoi, Vietnam}
\email{ngdhop@gmail.com}

\author{Ngo Viet Trung}
\address{International Centre for Research and Postgraduate Training\\ Institute of Mathematics \\ Vietnam Academy of Science and Technology, 18 Hoang Quoc Viet \\ Hanoi, Vietnam}
\email{nvtrung@math.ac.vn}

\begin{abstract}
This paper addresses the problem of comparing minimal free resolutions of symbolic powers of an ideal. Our investigation is focused on the behavior of the function $\depth R/I^{(t)} = \dim R - \pd I^{(t)} - 1$, where $I^{(t)}$ denotes the $t$-th symbolic power of a homogeneous ideal $I$ in a noetherian polynomial ring $R$ and $\pd$ denotes the projective dimension. 

It has been an open question whether the function $\depth R/I^{(t)}$ is non-increasing if $I$ is a squarefree monomial ideal. We show that $\depth R/I^{(t)}$ is almost non-increasing in the sense that $\depth R/I^{(s)} \ge \depth R/I^{(t)}$ for all $s \ge 1$ and $t \in E(s)$, where
$$E(s) = \bigcup_{i \ge 1}\{t \in \NN|\ i(s-1)+1 \le t \le is\}$$
(which contains all integers $t \ge (s-1)^2+1$). The range $E(s)$ is the best possible since we can find squarefree monomial ideals $I$ such that $\depth R/I^{(s)} < \depth R/I^{(t)}$ for $t \not\in E(s)$, which gives a negative answer to the above question.

Another open question asks whether the function $\depth R/I^{(t)}$ is always constant for $t \gg 0$.  We are able to construct counter-examples to this question by monomial ideals. On the other hand,
we show that if $I$ is a monomial ideal such that $I^{(t)}$ is integrally closed for $t \gg 0$ (e.g. if $I$ is a squarefree monomial ideal), then $\depth R/I^{(t)}$ is constant for $t \gg 0$ with  
$$\lim_{t \to \infty}\depth R/I^{(t)} = \dim R - \dim \oplus_{t \ge 0}I^{(t)}/\mm I^{(t)}.$$  

Our last result (which is the main contribution of this paper) shows that for any positive numerical function $\phi(t)$ which is periodic for $t \gg 0$, there exist a polynomial ring $R$ and a homogeneous ideal $I$ such that $\depth R/I^{(t)} = \phi(t)$ for all $t \ge 1$. As a consequence, for any non-negative numerical function $\psi(t)$ which is periodic for $t \gg 0$, there is a homogeneous ideal $I$ and a number $c$ such that $\pd I^{(t)} = \psi(t) + c$ for all $t \ge 1$.
\end{abstract}

\subjclass[2010]{13C15, 14B05}
\keywords{symbolic power, projective dimension, depth, asymptotic behavior, monomial ideal, 
integrally closed ideal, degree complex, local cohomology, Bertini-type theorem, system of linear diophantine inequalities}
\maketitle

%%%%%%%%%%%%%%%%%%%%%%%%%%%%

\section*{Introduction}
\label{sect_intro}

Throughout this paper, let $R$ be a Noetherian polynomial ring over a field $k$ and $I$ a homogeneous ideal of $R$. For every integer $t \ge 0$, the $t$-th {\it symbolic power} $I^{(t)}$ is defined as the intersection of the primary components of $I^t$ associated with the minimal primes of $I$.
When $I$ is the defining ideal of a reduced affine scheme $V$ over an algebraically closed field of characteristic zero, Zariski and Nagata showed that $I^{(t)}$ is the set of polynomials whose partial derivatives of orders up to $t-1$ vanish on $V$ (see e.g. Eisenbud and Hochster \cite{EH}). Their result gives a beautiful geometric interpretation of the symbolic powers.  

It is usually very difficult to study the behavior of symbolic powers. One of the reasons is that the {\em symbolic Rees algebra}
$$R_s(I) := \bigoplus_{t \ge 0}I^{(t)}$$
is in general  not finitely generated (see e.g. Roberts \cite{Ro}, Huneke \cite{Hu}, or Cutkosky \cite{Cu}). 
Most of the results until now have dealt with containments between symbolic and ordinary powers, 
which are initiated by works of Eisenbud and Mazur \cite{EM}, Ein, Lazarsfeld and Smith \cite{ELS}, and Hochster and Huneke \cite{HoHu}. 

It is of great interest to know whether there are relationships between the minimal free resolutions of different symbolic powers of $I$. The first important invariant of the minimal free resolution of an $R$-module $M$ is its length, which equals the projective dimension $\pd(M)$ of $M$.  By the Auslander-Buchsbaum formula, we have
$$\pd I^{(t)} = \dim R - \depth R/I^{(t)} - 1.$$
Since the depth can be characterized by other means, it is relatively easier to investigate the depth than the projective dimension. The aim of this paper is to study behavior of the function $\depth R/I^{(t)}$, $t \ge 1$.

Our investigation is inspired by results on the function $\depth R/I^t$, $t \ge 1$.
Due to Brodmann \cite{Br}, this function is convergent, i.e. $\depth R/I^t$ is constant for $t \gg 0$. 
In general, $\depth R/I^t$ tends to be a non-increasing function \cite{HH, HQ}. 
This was conjectured to be true for all squarefree monomial ideals 
until a counter-example was found by graph theorists \cite{KSS}.  
Step by step, one has realized that the function $\depth R/I^t$ can behave arbitrarily \cite{BHH,HNTT,HH,MST,MV}.  
Herzog and Hibi \cite{HH} conjectured that for any convergent non-negative numerical function $\phi(t)$, there exists a homogeneous ideal $I$ such that $\depth R/I^t = \phi(t)$ for all $t \ge 1$. 
This conjecture was recently settled in the affirmative \cite{HNTT}.  \par

To study the function $\depth R/I^{(t)}$ is harder than the function $\depth R/I^t$ because of the subtle nature of the symbolic powers. 
It has been an open question whether the function $\depth R/I^{(t)}$ is non-increasing if $I$ is a squarefree monomial ideals. This question has a positive answer for several classes of squarefree monomial ideals \cite{CPSTY, HKTT, KTY, TT2}.
We can show that the function $\depth R/I^{(t)}$ is almost non-increasing if $I$ is a monomial ideal such that $I^{(t)}$ is integrally closed for all $t \ge 1$ (that condition is satisfied if $I$ is squarefree). This is a consequence of the following result.
\medskip

\noindent {\bf Theorem \ref{decreasing}.}
{\em Let $I$ be a monomial ideal and $s \ge 1$ such that $I^{(s)}$ is integrally closed. Then \par
{\rm (i)} $\depth R/I^{(s)} \ge \depth R/I^{(st)}$ for all $t \ge 1$,\par
{\rm (ii)} there is a constant $a$ such that $\depth R/I^{(s)}\! \ge \depth R/I^{(t)}$ for $t \ge as^2$.}
\medskip

The proof of Theorem \ref{decreasing} involves techniques from combinatorial topology and linear programming. 
For squarefree monomial ideals, we can even determine the set of the numbers $t$ for which $\depth R/I^{(s)} \ge \depth R/I^{(t)}$. Set 
$$E(s) = \bigcup_{i \ge 1}\{t \in \NN|\ i(s-1)+1 \le t \le is\}.$$
Note that $E(s)$ contains all integers $t \ge (s-1)^2+1$. In particular, $E(s)$ is the set of all intergers $t \ge s$ 
only for $s = 1,2$.
\medskip

\noindent {\bf Theorems \ref{squarefree} and \ref{prop_example_increasing}.}
{\em 
For all $s \ge 1$, $E(s)$ is exactly the set of the exponents $t$ for which the inequality $\depth R/I^{(s)} \ge \depth R/I^{(t)}$ holds for every squarefree monomial ideal $I$.}
\medskip

Theorem \ref{squarefree} displays an unusual behavior of the depth of the symbolic powers, namely that 
$\depth R/I^{(s)} \ge \depth R/I^{(t)}$ for $t \in E(s)$, which is a union of disjoint intervals if $s \ge 3$. 
For instance, $\depth R/I^{(3)} \ge \depth R/I^{(t)}$ for $t \ge 5$ and $\depth R/I^{(4)} \ge \depth R/I^{(t)}$ for $t = 7,8$ and $t \ge 10$. 

Theorem \ref{prop_example_increasing} shows that for each $s \ge 3$, there exists a squarefree monomial ideal $I$ such that $\depth R/I^{(s)} < \depth R/I^{(t)}$ if and only if $t \not\in E(s)$.
In particular, this gives a negative answer to the open question whether the function $\depth R/I^{(t)}$ is always non-increasing. 

It is known that the symbolic Rees algebra $R_s(I)$ of a monomial ideal $I$ is finitely generated \cite{HHT}.  Using this fact one can deduce that the function $\depth R/I^{(t)}$ is {\em asymptotically periodic}, i.e. periodic for $t \gg 0$ (Proposition \ref{Rees}). However, 
it has been an open question whether $\depth R/I^{(t)}$ is always a convergent function \cite{HKTT}. \par

We will show that this question has a positive answer for a very large class of monomial ideals (including all squarefree monomial ideals) and that there is a formula for the asymptotic value in terms of the symbolic fiber ring
$$F_s(I) := \bigoplus_{t \ge 0} I^{(t)}/\mm I^{(t)},$$
where $\mm$ is the maximal homogeneous ideal of $R$. 
\medskip

\noindent {\bf Theorem \ref{constant}.}
{\em Let $I$ be a monomial ideal in $R$ such that $I^{(t)}$ is integrally closed for $t \gg 0$. 
Then $\depth R/I^{(t)}$ is a convergent function with 
$$\lim_{t \to\infty} \depth R/I^{(t)} = \dim R - \dim F_s(I),$$
which is also the minimum of $\depth R/I^{(t)}$ among all integrally closed symbolic powers $I^{(t)}$.}
\medskip

The convergence of the function $\depth R/I^{(t)}$ can be also deduced from results of \cite{HT1}, which investigated the Castelnuovo-Mumford regularity of integral closures of ideals of the form $I_1^t \cap \cdots \cap I_p^t$, where $I_1,...,I_p$ are monomial ideals. However, \cite{HT1} did not provide any information on $\lim_{t\to\infty} \depth R/I^{(t)}$. 
If $I$ is a squarefree monomial ideal, we immediately obtain the formula
$$\lim_{t \to\infty} \depth R/I^{(t)} = \min_{t \ge1} \depth R/I^{(t)} =  \dim R - \dim F_s(I),$$
which was proved recently in \cite{HKTT}. 
We also show that $I^{(t)}$ is often integrally closed for $t \gg 0$ 
if $I$ is the intersection of primary ideals generated by forms of the same degree (Proposition \ref{primary}).
 \par
 
For simplicity, we call a numerical function $\phi(t)$  a {\em symbolic depth function} over $k$ if there exist a polynomial ring $R$ over $k$ and a relevant homogeneous ideal $I \subset R$ such that $\depth R/I^{(t)} = \phi(t)$ for $t \ge 1$. Note that a symbolic depth function is always a positive numerical function. \par

We establish a method of constructing symbolic depth functions which only take the values 1 and 2 depending on a single ideal-theoretical containment (Proposition \ref{lem_key}). To check this containment amounts to solving systems of linear diophantine inequalities. Using this method we construct monomial ideals whose symbolic depth functions are not convergent, thereby giving a negative answer to the afore mentioned open question:
\medskip

\noindent {\bf Theorem \ref{periodic}.} 
{\em Let $m\ge 2$ and $0 \le d < m$ be integers. 
There exists a monomial ideal $I$ in $R=k[x,y,z,u,v]$ such that
$$
\depth R/I^{(t)}= \begin{cases}
2 &\text{if $t \equiv d$ modulo $m$},\\
1 &\text{otherwise}.
\end{cases}
$$}

The next problem is how wild a symbolic depth function could be. 
Surprisingly, we can show that any asymptotically periodic positive numerical function is a symbolic depth function.  
The idea is to construct ideals with basic symbolic depth functions 
and to build up any asymptotically periodic positive numerical function by using closed operations
within the class of symbolic depth functions. \par

First, there is a simple way to obtain from two symbolic depth functions $\phi(t)$ and $\psi(t)$ a new symbolic depth function $\phi(t)+\psi(t)+1$ (Proposition \ref{product}).  
However, the new symbolic depth function has higher values. 
For instance, we always have $\phi(t)+\psi(t)+1 \ge 3$ for all $t \ge 1$.
Therefore, we need a Bertini-type theorem to reduce the depth of symbolic powers. 
For a homogeneous ideal $I$ with $\depth R/I^{(t)} \ge 2$ for all $t \ge 1$, we have to find a linear form $f \in R$ such that $f$ is a regular element on $I^{(t)}$ and if we set $S = R/(f)$ and $Q = (I,f)/(f)$, then  
$$S/Q^{(t)} = R/(I^{(t)},f)$$
for all $t \ge 1$.
There is an obstacle in finding such a linear form, namely that $f$ has to be
the same element for all symbolic powers $I^{(t)}$, which form an infinite family of ideals. 
However, using a generic linear form in a polynomial ring over a purely transcendental extension of $k$ we can prove such a Bertini-type theorem (Proposition \ref{Bertini}). \par

Using the operation $\phi(t)+\psi(t)$ and the above Bertini-type theorem we can build up any asymptotically periodic positive numerical function from some basic symbolic depth functions. 
In this way we obtain the following result, which is the main contribution of this paper:
\medskip

\noindent {\bf Theorem \ref{ubiquity}.}
{\em Let $\phi(t)$ be an arbitrary asymptotically periodic positive numerical function.
Given a field $k$, there exist a polynomial ring $R$ over a purely transcendental extension of $k$ and a
homogeneous ideal $I \subset R$ such that $\depth R/I^{(t)}=\phi(t)$ for $t \ge 1$.}
\medskip

The hardest part of the proof of Theorem \ref{ubiquity} is to construct ideals with basic symbolic depth functions.
The constructions contain new ideas and techniques, which can be used to find ideals whose symbolic powers have irregular behavior with respect to other homological invariants.  
\medskip

\noindent {\bf Theorem \ref{pd}.} 
{\em Let $\psi(t)$ be an arbitrary asymptotically periodic non-negative numerical function and $m = \max_{t \ge 1}\psi(t)$. 
Given a field $k$, there is a number $c$ such that there exist a polynomial ring $R$ 
in $m+c+2$ variables over a purely transcendental extension of $k$ and a
homogeneous ideal $I \subset R$ for which $\pd I^{(t)} = \psi(t) + c$ for $t\ge 1$.}
\medskip

It is of great interest to know all possible functions of the projective dimension of symbolic powers of a homogeneous ideal. For that one needs to compute the smallest number $c$ of Theorem \ref{pd} for each function $\psi(t)$.
This number is determined by the smallest number of variables of a polynomial ring $R$ which contains a homogeneous ideal with a given symbolic depth function.
However, we are not able to compute that number. 
\medskip

Theorem \ref{ubiquity} leaves some questions unanswered. 
We begin with monomial ideals having basic symbolic depth functions over $k$ and end up 
with non-monomial ideals having any given asymptotically periodic symbolic depth function over a purely transcendental extension of $k$. That is unlike the case of ordinary powers, where any convergent non-negative numerical function is the depth function of a monomial ideal over $k$ \cite{HNTT}. 
This leads us to the following question:
\medskip

\noindent {\bf Question \ref{field}.}
Given any asymptotically periodic positive numerical function $\phi(t)$, do there exist a polynomial ring $R$ over any field $k$ and a monomial ideal $I \subset R$ such that $\depth R/I^{(t)} = \phi(t)$ for all $t \ge 1$?
\medskip

Finally, we would like to point out that there is no known example of a homogeneous ideal whose symbolic depth function is not asymptotically periodic. To find such an ideal is a hard problem because its symbolic Rees algebra has to be non-Noetherian, whose existence is related to Hilbert's fourteenth problem \cite{Ro}.
It is not clear whether such an ideal exists at all. \par

This paper is organized as follows. 
Section 1 prepares results on the depth of monomial ideals, which will be used to estimate the depth of symbolic powers.
In Section 2 we compare the depth of an integrally closed symbolic power with that of higher symbolic powers.
Section 3 investigates monomial ideals whose symbolic depth functions are convergent.
In Section 4 we construct ideals with basic symbolic depth functions. 
Section 5 shows how to manipulate existing symbolic depth functions to obtain new ones.
Section 6 is devoted to the proof that any asymptotically periodic positive numerical function is a symbolic depth function. \par

For unexplained notions and standard facts in commutative algebra we refer the reader to \cite{BH, HS}. 
\medskip

\noindent{\bf Acknowledgement}.
Hop Dang Nguyen is partially supported by Project CT 0000.03/19-21 of Vietnam Academy of Science and Technology.
Ngo Viet Trung is partially supported by Vietnam National Foundation for Science and Technology Development. Part of this work was done during research stays of the authors at Vietnam Institute for Advanced Study in Mathematics. The authors would like to thank Huy T\`ai H\`a and Tran Nam Trung for their collaboration on the joint paper \cite{HNTT} which initiated this work. They are also grateful to the referee for many suggestions which help improve the presentation of the paper. After the revision of this paper, the authors have been informed that Theorem 2.7  has been recently obtained in a modified form by different methods by J. Montano and L. Nunez-Betancourt (arXiv 1809.02308) and S. A. Seyed Fakhari  (arXiv 1812.03742).

%%%%%%%%%%%%%%%%%%%%%%%%%%%%%%

\section{Depth of monomial ideals}
\label{sect_convergent_monomial}

Let $R = k[x_1,...,x_n]$ be a polynomial ring over a field $k$ and $I$ a monomial ideal in $R$.  
Let $H_\mm^i(R/I)$ denote the $i$-th local cohomology of $R/I$ with support at the maximal homogeneous ideal $\mm$ of $R$.
It is well-known that
$$\depth R/I = \min\{i \mid H_\mm^i(R/I) \neq 0\}.$$
Since $R/I$ has a natural $\NN^n$-graded structure, 
the local cohomology module $H_\mm^i(R/I)$ has a $\ZZ^n$-graded structure.
We have the following formula on the dimension of the graded component $H_\mm^i(R/I)_\a$, $\a \in \ZZ^n$.
 
For $\a = (a_1,...,a_n) \in \ZZ^n$, set $x^\a = x_1^{a_1} \cdots x_n^{a_n}$.
We denote by $G_\a$ the negative support of $\a$, i.e. $G_\a := \{i \in [n] ~\big|~ a_i < 0\}$, 
where $[n] = \{1,...,n\}$.
For every subset $F \subseteq [n]$, let $R_F = R[x_i^{-1}|\ i \in F]$. Define
\begin{equation*}
\D_\a(I)  := \{F \setminus G_\a|\ G_\a \subseteq F,\  x^\a \not\in IR_F\},
\end{equation*}
which is a simplicial complex on the vertex set $[n]$. We call it the \emph{degree complex} of $I$ with respect to $\a$. \par

\begin{thm} \label{Takayama}  \cite[Theorem 1]{Ta}
$\dim_kH_\mm^i(R/I)_\a =  \dim_k\widetilde H_{i-|G_\a|-1}(\D_\a(I),k).$
\end{thm}

The above definition of a degree complex is simpler than the original construction in \cite{Ta}. 
Moreover, the original result contains additional conditions on $\a$.  
However, the original proof shows that we may drop these conditions, 
which is more convenient for our investigation. See \cite{MT} for more details.

\begin{cor} \label{depth} 
$\depth R/I =  \min\{|G_\a| +j|\ \a \in \ZZ^n, j \ge 0, \widetilde H_{j-1}(\D_\a(I),k) \neq 0\}.$
\end{cor}

\begin{proof} 
By Theorem \ref{Takayama}, we have 
$$\depth R/I =  \min\{i|\ \text{there is } \a \in \ZZ^n \text{ such that } \widetilde H_{i-|G_\a|-1}(\D_\a(I),k) \neq 0\}.$$
Replacing $i$ by $|G_\a|+j$, we obtain the assertion.
\end{proof}

To compare the depths of two monomial ideals, one only need to compare their degree complexes.

\begin{prop} \label{compare}
Let $I$ and $J$ be two monomial ideals in $R$ such that for every $\a \in \ZZ^n$, there exists a vector $\b \in \ZZ^n$ with $G_\b \subseteq G_\a$ such that $\D_\b(J) = \D_\a(I)$. Then $\depth R/I \ge \depth R/J$.
\end{prop}

\begin{proof}
By Corollary \ref{depth}, we may assume that $\depth R/I =  |G_\a|+i$ for some $\a \in \ZZ^n$ and $i \ge 0$ such that  $\widetilde H_{i-1}(\D_\a(I),k) \neq 0$. Choose $\b \in \ZZ^n$ as in the assumption. Then $\depth R/I \ge |G_\b|+ i \ge \depth R/J$, where the last inequality follows from Corollary \ref{depth}.
\end{proof}

It is clear that $\depth R/I \ge 1$ if and only if $\mm$ is not an associated prime of $I$.
For  $\depth R/I \ge 2$ we have the following criterion.

\begin{prop} \label{depth2}
$\depth R/I \ge 2$ if and only if the following conditions are satisfied:\par
 {\rm (i)} $\depth R/I \ge 1$,\par
{\rm (ii)} $\depth R_j/I_j \ge 1$ for all $j = 1,...,n$, where $R_j = k[x_i|\ i \neq j]$ and $I_j = IR[x_j^{-1}] \cap R_j$. \par
{\rm (iii)}  Every degree complex $\D_\a(I)$ with $\a \in \NN^n$ is connected.
\end{prop}

\begin{proof}
It is well known that (i) means $H_\mm^0(R/I) = 0$. It is also known that (ii) and (iii) are equivalent to the condition $H_\mm^1(R/I) = 0$ \cite[Proposition 1.6]{TT}. Therefore, the conclusion follows from the fact that $\depth R/I \ge 2$ if and only if $H_\mm^i(R/I) = 0$ for $i = 0,1$.
\end{proof}

Now we will address the problem of computing degree complexes. 
By definition, every face of a degree complex $\D_\a(I)$ is of the form $F \setminus G_\a$ for a subset $F \subseteq [n]$, 
$G_\a \subseteq F$. 
For every subset $F \subseteq [n]$ we denote by $P_F$ the ideal of $R$ generated by the variables $x_i$,  $i \not\in F$. 

\begin{lem} \label{facet} 
Let $F \subseteq [n]$ such that $G_\a \subseteq F$. 
If $F \setminus G_\a$ is a facet of $\D_\a(I)$, then 
$P_F$ is an associated prime of $I$.
\end{lem}
 
\begin{proof}
Let $S = k[x_i \mid i \not\in F]$. Set
$Q = IR_F \cap  S$ and $\widetilde Q := \cup_{t \ge 1}(Q :\nn^t),$ where
$\nn$ denotes the maximal homogeneous ideal of $S$.
Let $\b$ denote the vector obtained from $\a$ by setting $a_i = 0$ for $i \in F$.
By \cite[Lemma 1.3]{TT}, $F \setminus G_\a$ is a facet of $\D_\a(I)$ if and only if $x^\b \in \widetilde Q \setminus Q$.
If $x^\b \in \widetilde Q \setminus Q$, then $\widetilde Q \neq Q$. By  \cite[Lemma 1.2]{HLT}\footnote{The notation of $P_F$ in \cite{HLT} is different.}, $\widetilde Q \neq Q$ if and only if $P_F$ is an associated prime of $I$.
\end{proof}

By Lemma \ref{facet}, to compute the facets of $\D_\a(I)$ we only need to look for the sets $F \subseteq [n]$ such that
$P_F$ is an associated prime of $I$, $G_\a \subseteq F$, and $x^\a \not\in IR_F$. 
This task becomes easier if $I$ is an {\em unmixed} ideal, 
i.e. every associated prime is a minimal prime of $I$.  \par

Let $\F(I)$ denote the set of all subsets $F \subseteq [n]$ such that $P_F$ is a minimal prime of $I$. 
If $F \in \F(I)$, we denote by $I_F$ the primary component of $I$ associated with $P_F$. 
Let $\a_+$ denote the vector obtained from $\a$ by setting every negative coordinate to zero.

\begin{prop} \label{unmixed} 
Let $I$ be an unmixed monomial ideal. 
Let $F \subseteq [n]$ such that $G_\a \subseteq F$. 
Then $F \setminus G_\a$ is a facet of $\D_\a(I)$ if and only if $F \in \F(I)$ and $x^{\a_+} \not\in I_F$.
\end{prop}

\begin{proof}
First, we show that $F \setminus G_\a$ is a facet of $\D_\a(I)$ if and only if $F \in \F(I)$ and $x^\a \not\in I_F$.
Without restriction we may assume that $F \setminus G_\a$ is a face of $\D_\a(I)$ or, equivalently, $x^\a \not\in IR_F$.
If $F \setminus G_\a$ is a facet of $\D_\a(I)$, then $F \in \F(I)$ by Lemma \ref{facet}.
If $F \setminus G_\a$ is not a facet of $\D_\a(I)$, then $F$ is properly contained in a set $G \subseteq [n]$ such that $G \setminus G_\a$ is a facet of $\D_\a(I)$. Since $G \in \F(I)$ and since there is no inclusion among the sets of $\F(I)$, $F \not\in \F(I)$.\par

It remains to show that for $F \in \F(I)$, $x^\a \not\in IR_F$ if and only if $x^{\a_+} \not\in I_F$.
It is obvious that $x^\a \in IR_F$ if and only if $x^{\a_+} \in IR_F$ if and only if $x^{\a_+} \in IR_F \cap R$.
Since $I$ is unmixed, $I = \cap_{G \in \F(I)} I_G$.
Since $R_F = R[x_i^{-1} \mid i \in F]$ and since $P_G$ is generated by the variables $x_i \not\in G$, 
we have $I_GR_F = R_F$ for all $G \neq F$ of $\F(I)$ and $I_FR_F \neq R_F$. 
Therefore, $IR_F \cap R = I_FR_F \cap R = I_F$, as desired.  
\end{proof}

One of the distinguished features of monomial ideals is the distributive property of addition over intersection:
$$(I_1 \cap I_2) + I_3 = (I_1 + I_3) \cap (I_2 + I_3).$$
This property allows us to estimate the depth of an intersection of monomials ideals from those of their sums.

\begin{lem} \label{lem_depth_intersectionandsum}
Let $I_1,\ldots,I_s$ be monomial ideals of $R$. Assume that there is an integer $\delta \ge 0$ such that
\[
\depth \frac{R}{I_{j_1}+\cdots+I_{j_i}} \ge s+\delta-i
\]
for  all $1\le i\le s$ and $1\le j_1< \cdots <j_i\le s$. Then 
\[
\depth \frac{R}{I_1 \cap \cdots \cap I_s} \ge s+\delta-1.
\]
\end{lem}

\begin{proof}
If $s=1$, there is nothing to do. If $s\ge 2$,
consider the exact sequence
\[
0\to \frac{R}{I_1 \cap  \cdots \cap I_s} \to \frac{R}{I_1 \cap \cdots \cap I_{s-1}} \oplus  \frac{R}{I_s} \to \frac{R}{(I_1 \cap  \cdots \cap I_{s-1})+I_s} \to 0.
\]
Then we have
\begin{align*}
& \depth \frac{R}{I_1 \cap \cdots \cap I_s}\\
& \ge \min  \left\{ \depth \frac{R}{ I_1 \cap \cdots \cap I_{s-1}}, \depth \frac{R}{I_s}, 
                                \depth \frac{R}{(I_1 \cap  \cdots \cap I_{s-1})+I_s}+1 \right\}.
\end{align*}
By induction on $s$, we may assume that 
$$\depth \frac{R}{ I_1 \cap \cdots \cap I_{s-1}} \ge (s-1)+(\delta+1)-1 = s+\delta-1.$$
By the assumption, $\depth R/I_s \ge s+\delta - 1$.
It remains to show that
\[
 \depth \frac{R}{(I_1 \cap  \cdots \cap I_{s-1})+I_s} \ge s+\delta-2.
\]

Set $I'_j=I_j+I_s$, $j = 1,...,s-1$. Then
$$(I_1 \cap  \cdots \cap I_{s-1})+I_s = I'_1 \cap  \cdots \cap I'_{s-1}.$$
For all $1\le i \le s-1$ and $1\le j_1< \cdots <j_i\le s-1$, we have
$$
\depth \frac{R}{I'_{j_1}+\cdots+I'_{j_i}} =  \depth \frac{R}{I_{j_1}+\cdots+I_{j_i} +I_s} \ge s+\delta - i-1.
$$
Using the induction hypothesis for $s-1$, we get
$$\depth  \frac{R}{I'_1 \cap  \cdots \cap I'_{s-1}} \ge s+\delta - 2.$$
\end{proof}

We can also estimate the depth of the intersection or the sum of two monomial ideals by looking for regular sequences modulo these ideals, which are sometimes easy to find because their associated primes are generated by variables.

\begin{lem} \label{regular}
Let $I_1$ and $I_2$ be two monomial ideals in $R$. 
Let $i_1,...,i_r$ and $j_1,...,j_r$ be two disjoint families of integers in $[n]$. Then \par
{\rm (i)} $x_{i_1}-x_{j_1},...,x_{i_r}-x_{j_r}$ form a regular sequence for $R/(I_1\cap I_2)$ if they form a regular sequence for $R/I_1$ and $R/I_2$. \par
{\rm (ii)} $x_{i_1}-x_{j_1},...,x_{i_r}-x_{j_r}$ form a regular sequence for $R/(I_1+I_2)$
if they form a regular sequence for $R/I_1$ and the minimal generators of $I_2$ are not divisible by any of the variables $x_{i_1},..,x_{i_r},x_{j_1},...,x_{j_r}$.  
\end{lem}

\begin{proof}
(i) The case $r = 1$ is trivial. If $r \ge 2$, we consider the polynomial ring $R'$ obtained from $R$ by removing the variable $x_{i_1}$. Let $I'_1,I'_2$ be the monomial ideals in $R'$ obtained from $I_1,I_2$ by the substitution $x_{i_1} \to x_{j_1}$. Then 
\begin{align*}
R' & \cong R/(x_{i_1}-y_{j_1}),\\
I'_t & \cong (I_t,x_{i_1}-y_{j_1})/(x_{i_1}-y_{j_1}),\ t = 1,2.
\end{align*}
Therefore, $x_{i_2}-x_{j_2},...,x_{i_r}-x_{j_r}$ form a regular sequence for 
$R'/I'_t$ for $t = 1,2$. By the induction hypothesis, $x_{i_2}-x_{j_2},...,x_{i_r}-x_{j_r}$ form a regular sequence for 
$R'/(I'_1 \cap I'_2)$. From this it follows that $x_{i_1}-x_{j_1},...,x_{i_r}-x_{j_r}$ form a regular sequence for 
$R/(I_1 \cap I_2)$.  \par
(ii) If $r = 1$, we have to show that $x_{i_1}-x_{j_1}$ does not belong to any associated prime $P$ of $I_1+I_2$. 
Note that the sum of two primary monomials ideals is again a primary ideal.
Then, using the distributive property of the intersection over the addition, we can see that
$P = P_1+P_2$, where $P_1$ and $P_2$ are associated primes of $I_1$ and $I_2$, respectively.
If $x_{i_1}-x_{j_1} \in P_1+P_2$, then $x_{i_1}, x_{j_1} \in P_1+P_2$. By the assumption of (ii), we have $x_{i_1}, x_{j_1} \not\in P_2$. Hence, $x_{i_1}, x_{j_1} \in P_1$, which implies $x_{i_1}- x_{j_1} \in P_1$, a contradiction. If $r \ge 2$, we use the induction hypothesis as in the above proof for (i).
\end{proof}

%%%%%%%%%%%%%%%%%%%%%%%%%

\section{Depth of integrally closed symbolic powers}

The aim of this section is to study the symbolic depth function of a squarefree monomial ideal.
More generally, we will compare the depth of a symbolic power, that is integrally closed, with depths of higher symbolic powers.

We keep the notation of the preceding section. 
Let $I$  be a monomial ideal.
Since $\F(I)$ is the set of all subsets $F \subseteq [n]$ such that 
$P_F$ is a minimal prime ideal of $I$ and $I_F$ is the associated $P_F$-primary component,
we have the following formula for the symbolic powers of $I$.

\begin{lem} \label{decomposition} \cite[Lemma 3.1]{HHT}
$I^{(t)} = \bigcap_{F \in \F(I)} I_F^t$.
\end{lem}

\begin{prop} \label{closed}
$I^{(t)}$ is integrally closed if and only if $I_F^t$ is integrally closed for all $F \in \F(I)$.
\end{prop}

\begin{proof}
Assume that $I^{(t)}$ is integrally closed. 
Then $I^{(t)}R_{P_F}$ is integrally closed for all $F \in \F(I)$ by \cite[Proposition 1.1.4(3)]{HS}. 
By Lemma \ref{decomposition}, we have $I_F^t= I^{(t)}R_{P_F}\cap R$. 
Hence, $I_F^t$ is integrally closed by \cite[Proposition 1.6.2]{HS}. 
Conversely, assume that $I_F^t$  is integrally closed for all $F \in \F(I)$.
Since the intersection of integrally closed ideals is integrally closed, $I^{(t)}$ is integrally closed 
by Lemma \ref{decomposition}.
\end{proof}

Let $\overline{I}$ denote the integral closure of $I$.
It is well known that $x^\a \in \overline{I}$ if and only if $x^{t\a} \in I^t$ for some $t \ge 1$ (see e.g. \cite[\textsection 1.4]{HS}).  From this it follows that if $I$ is integrally closed, then $x^\a \in I$ if and only if $x^{t\a} \in I^t$ for some (or all) $t \ge 1$. As we shall see below, this property allows us to compare the depth of an integrally closed unmixed monomial ideal with those of its symbolic powers.

\begin{prop} \label{integral}
If $I$ is an integrally closed unmixed monomial ideal, then 
$\depth R/I \ge \depth R/I^{(t)}$ for all $t \ge 1$. 
\end{prop}

\begin{proof}
By Proposition \ref{compare}, it suffices to show that $\D_\a(I) = \D_{t\a}(I^{(t)})$ for all $\a \in \ZZ^n$.
Let $F \subseteq [n]$ such that $G_\a \subseteq F$.
By Proposition \ref{unmixed}, 
$F \setminus G_\a$ is facet of $\D_\a(I)$ if and only if $F\in \F(I)$ and 
$x^{\a_+} \not\in I_F$. 
Since $I$ and $I^{(t)}$ share the same minimal primes, $\F(I) = \F(I^{(t)})$.
By Proposition \ref{unmixed} and Lemma \ref{decomposition}, $F \setminus G_\a$ is a facet of $\D_{t\a}(I^{(t)})$ if and only if $F \in \F(I)$ and $x^{t\a_+} \not\in I_F^t$.

Since $I$ is unmixed, $I^{(1)} = I$. 
Since $I^{(1)}$ is integrally closed, $I_F$ is integrally closed by Proposition \ref{closed}.
As observed above, $x^{\a_+} \in I_F$ if and only if $x^{t\a_+} \in I_F^t$.  
Therefore, we can conclude that $\D_\a(I) = \D_{t\a}(I^{(t)})$.
\end{proof}

We shall see that if $I^{(s)}$ is integrally closed for some $s \ge 1$, then $\depth R/I^{(s)} \ge  \depth R/I^{(t)}$ for almost all $t > s$. For that we shall need the following membership criteria for $I^t$ and its integral closure $\overline{I^t}$, 
which were first presented in \cite[\textsection 3]{Tr1},  \cite[\textsection 2]{Tr2} without proofs and recently in \cite{HTr}.\par

Let $I = (x^{\a_1},...,x^{\a_m})$, where $\a_1,...,\a_m \in \NN^n$.
Let $M_I$ be the matrix whose columns are $\a_1,...,\a_m$.
For any vector $\v = (v_1,...,v_m)$ we set $|\v| = v_1 + \cdots + v_m$.
For every $\a \in \NN^n$, we define
$$\nu_\a(I) := \max\{|\v| \mid \v \in \NN^m, \v \cdot M_I \le \a\},$$
where the inequality is taken componentwise. 
Note that  $\nu_{t\a}(I) \ge t\nu_\a(I)$ for all $t \ge 1$. 

\begin{lem} \label{member} \cite[Proposition 3.1(i)]{Tr1} \cite[2.1.2]{Tr2}
$x^\a \in I^t$ if and only if $\nu_\a(I) \ge t$.
\end{lem}

\begin{proof}
It is clear that $x^\a \in I^t$ if and only if there exist integers $v_1,...,v_m \ge 0$ with $v_1 + \cdots + v_m = t$
and $\b \in \NN^n$ such that
$$\a = v_1\a_1 + \cdots + v_m\a_m + \b.$$
This condition means that there exist $\v = (v_1,...,v_m) \in \NN^m$ with $|\v| = t$ such that $\v \cdot M_I \le \a$.
Therefore, the conclusion follows from the definition of $\nu_\a(I)$.
\end{proof}

We approximate $\nu_\a(I)$ by the number
$$\nu_\a^*(I) := \max\{|\v| \mid \v \in \RR_+^m, \v \cdot M_I \le \a\},$$
where $\RR_+$ denotes the set of non-negative real numbers.
Note that we always have $\nu_\a^*(I) \ge \nu_\a(I)$ and $\nu_{t\a}^*(I) = t\nu_\a^*(I)$ for all $t \ge 1$. 
 
The number $\nu_\a^*(I)$ can be computed by linear programming. Let
$$N_\a(I) := \{\v \in \RR_+^m \mid \v \cdot M_I \le \a\}.$$
Then $N_\a(I)$ is a rational convex polyhedron. 
It is well known that $\nu_\a^*(I) = |\v|$ for some vertex $\v$ of $N_\a(I)$ (see e.g. \cite{Sc}). 
Every vertex of $N_\a(I)$ is the solution of equations of the system $\v \cdot M_I = \a$.
By Cramer's rule, there exists a positive integer $q$ depending only on the matrix $M_I$ (not on $\a$) such that $q\v \in \NN^m$ for all vertices $\v$ of $N_\a(I)$ for all $\a \in \NN^n$.
\par
 
\begin{lem} \label{closure} \cite[Proposition 3.1(ii)]{Tr1} \cite[2.1.4]{Tr2}
$x^\a \in \overline{I^t}$ if and only if $\nu_\a^*(I) \ge t$.
\end{lem}

\begin{proof}
Let $x^\a \in \overline{I^t}$. Then $x^{s\a} \in I^{st}$ for some $s \ge 1$.
Therefore, $\nu_{s\a}(I) \ge st$. Hence, $s\nu_\a^*(I) = \nu_{s\a}^*(I) \ge st$.
This implies $\nu_\a^*(I) \ge t$.

Conversely, assume that $\nu_\a^*(I) \ge t$. Let $\v$ be a vertex of $N_\a(I)$
such that $\nu_\a^*(I) = |\v|$. Let $q$ be a positive integer such that $q\v \in \NN^m$.
Since the vertices of $N_{q\a}(I)$ are the $q$-multiples of the vertices of $N_\a(I)$, we have
$\nu_{q\a}^*(I) = |q\v|$. Since $q\v \in \NN^m$, this implies $\nu_{q\a}(I) = |q\v| = q\nu_\a^*(I) \ge qt$.
Therefore, $x^{q\a} \in I^{qt}$. Hence, $x^\a \in \overline{I^t}$.
\end{proof}

Using the above lemmas we show that if a symbolic power $I^{(s)}$ is integrally closed, then $\depth R/I^{(s)} \ge \depth R/I^{(t)}$ for $t \gg 0$. 

\begin{thm} \label{decreasing}
Let $I$ be a monomial ideal such that $I^{(s)}$ is integrally closed for some $s \ge 1$. Then \par
{\rm (i)} $\depth R/I^{(s)} \ge \depth R/I^{(st)}$ for all $t \ge 1$. \par
{\rm (ii)} There is a constant $a$ such that $\depth R/I^{(s)} \ge \depth R/I^{(t)}$ for $t \ge as^2$.
\end{thm}

\begin{proof}
By Lemma \ref{decomposition}, we have $I^{(st)} = (I^{(s)})^{(t)}$ for all $t \ge 1$.
Therefore, (i) follows from Proposition \ref{integral}. In particular, $\depth R/I^{(1)} \ge \depth R/I^{(t)}$ for all $t \ge 1$. 
Therefore, to prove (ii), we may assume that $s \ge 2$. We will show that there exists an integer $q > 0$ (independent of $s$) such that $\depth R/I^{(s)} \ge \depth R/I^{(t)}$ if $t = c(qs - 1) + r$ for some $c \ge 1$, $1 \le r \le qs-1$ and $r \le c$. This implies (ii) because this condition on $t$ is satisfied if $t \ge (qs-1)^2+1$. \par 

First, from the observation on the polyhedron $N_\a(I)$ before Lemma \ref{closure} we can see that there exists  a positive integer $q$ such that  $q\v$ is an integral vector for all vertices $\v$ of all polyhedra $N_\a(I_F)$, $F \in \F(I)$, $\a \in \NN^n$. 
Let $t = c(qs - 1) + r$ as above.
We will show that $\D_{cq\a}(I^{(t)}) = \D_\a(I^{(s)})$ for all $\a \in \ZZ^n$, which then implies $\depth R/I^{(s)} \ge \depth R/I^{(t)}$ by Proposition \ref{compare}. \par

Let $F \subseteq [n]$ such that $G_\a \subseteq F$.
By Proposition \ref{unmixed} and Lemma \ref{decomposition}, $F \setminus G_\a$ is a facet of $\D_\a(I^{(s)})$ if and only if 
$F \in \F(I)$ and $x^{\a_+} \not\in I_F^s$.
Since $I^{(s)}$ is integrally closed, $I_F^s$ is integrally closed by Proposition \ref{closed}.  
By Lemma \ref{closure}, $x^{\a_+} \in I_F^s$ if and only if $\nu_{\a_+}^*(I_F) \ge s$. 
Therefore, $F \setminus G_\a$ is a facet of $\D_\a(I^{(s)})$ if and only if $F \in \F(I)$ and $\nu_{\a_+}^*(I_F) < s$.
\par

Similarly as above, we can use Proposition \ref{unmixed} and Lemma \ref{member} to show that
$F \setminus G_\a$  is a facet of $\D_{cq\a}(I^{(t)})$ if and only if $F \in \F(I)$ and $\nu_{cq\a_+}(I_F) < t$. \par

Assume that $F \in \F(I)$. 
If $F \setminus G_\a$  is a facet of $\D_\a(I^{(s)})$, then $\nu_{\a_+}^*(I_F) < s$. 
We know that $\nu_{\a_+}^*(I_F) = |\v|$ for a vertex $\v$ of $N_{\a_+}(I_F)$.
By the choice of $q$, we have $q\nu_{\a_+}^*(I_F) = q|\v| \in \NN$. Hence
 $\nu_{\a_+}^*(I_F) \le s - 1/q$. From this it follows that
$$\nu_{cq\a_+}(I_F) \le \nu^*_{cq\a_+}(I_F) = cq\nu_{\a_+}^*(I_F) = cq(s-1/q) = c(qs-1) < t.$$
Therefore, $F \setminus G_\a$  is a facet of $\D_{cq\a}(I^{(t)})$. 
If $F \setminus G_\a$  is not a facet of $\D_\a(I^{(s)})$, then $\nu_{\a_+}(I_F) \ge s$. Hence,
$$\nu_{cq\a_+}(I_F) \ge  cq\nu_{\a_+}(I_F) \ge cqs = c(qs-1) + c \ge c(qs-1)+ r = t.$$
From this it follows that $F \setminus G_\a$  is not a facet of $\D_{cq\a}(I^{(t)})$. 
So we can conclude that $\D_{cq\a}(I^{(t)}) = \D_\a(I^{(s)})$, as required.
\end{proof}

If $I$ is a squarefree monomial, we know that $I^{(s)}$ is integrally closed for all $s \ge 1$ (see the proof below).
In this case, we can say more about the set of integers $t$  for which $\depth R/I^{(s)} \ge \depth R/I^{(t)}$. 
For all $s \ge 1$, set 
$$E(s) = \bigcup_{i \ge 1}\{t \in \NN|\ i(s-1)+1 \le t \le is\}.$$
Note that for $s \ge 2$, $E(s)$ is exactly the set of all integers $t$ of the form $t = c(s - 1) + r$ for some $c \ge 1$, $1 \le r \le s-1$ and $r \le c$. 

\begin{thm} \label{squarefree}
Let $I$ be an arbitrary squarefree monomial ideal. For all $s \ge 1$,
$\depth R/I^{(s)} \ge \depth R/I^{(t)}$ if $t \in E(s)$. 
\end{thm}

\begin{proof}
Since $I$ is squarefree, $I_F = P_F$ for all $F \in \F(I)$.
By Lemma \ref{decomposition},  $I^{(s)} = \cap_{F \in \F(I)} P_F^s$ for all $s \ge 1$.
It is clear that $P_F^s$ is integrally closed. Hence, $I^{(s)}$ is integrally closed by Proposition \ref{closed}. \par

By Theorem \ref{decreasing}(i) we may assume that $s \ge 2$.
Let $\e_i$ denote the $i$-th unit vector. 
Then $P_F = (x^{\e_i}\mid i \not\in F)$. 
Hence the columns of the matrix $M_{P_F}$ are unit vectors.
Note that every vertex $\v$ of $N_\a(P_F)$ is a solution of the system $\v \cdot M_{P_F} = \a$.
Using Cramer's rule we can see that the vertices of $N_\a(P_F)$ 
are integral points for all $\a \in \NN^n$. 
Therefore, we may choose $q = 1$ in the proof of Theorem \ref{decreasing}.
Then $\depth R/I^{(s)} \ge \depth R/I^{(t)}$ if $t = c(s - 1) + r$ for some $c \ge 1$, $1 \le r \le s-1$ and $r \le c$. 
Since $E(s)$ is exactly the set of all such integers $t$, this gives the conclusion.
\end{proof}

It has been an open question whether the function $\depth R/I^{(t)}$ is non-increasing for all squarefree monomial ideals $I$. This question has a positive answer for cover ideals of arbitrary graphs, i.e. unmixed height 2 squarefree monomial ideals \cite[Theorem 3.2]{HKTT} (see \cite{CPSTY} for the case of bipartite graphs) and for edge ideals of very well covered graphs or graphs with leaves \cite[Theorem 5.2]{KTY}. The similar question for ordinary powers was raised as a conjecture in \cite[p. 535]{HH}, which has been recently settled in the negative \cite{KSS}. 

The following result shows that the range $E(s)$ in Theorem \ref{squarefree} is the best possible.
Since $E(s)$ is not the set of all integers $\ge s$ for $s \neq 1,2$, this result also gives a negative answer to the above question.

\begin{thm} \label{prop_example_increasing}
For each $s \ge 3$, there exists a squarefree monomial ideal $I$ such that
$\depth R/I^{(s)} < \depth R/I^{(t)}$ if and only if $t\not\in E(s)$.
\end{thm}

\begin{proof}
Let $R=k[x_{i,j}|\ 1\le i\le s, 1\le j\le s-1]$ and
$$I = P_1 \cap \cdots \cap P_s \cap Q,$$
where
\begin{align*}
P_i &= (x_{i,1},x_{i,2},\ldots,x_{i,s-1}),\ i = 1,...,s,\\
Q & = \bigcap_{1\le j_1,\ldots,j_s \le s-1} (x_{1,j_1},\ldots,x_{s,j_s}).
\end{align*}
We will show that $\depth R/I^{(t)} = s-1$ if $t \in E(s)$ and $\depth R/I^{(t)} \ge  s$ if $t \not\in E(s)$. 
Since $s \in E(s)$, this implies the statement of Theorem \ref{prop_example_increasing}.  \par

It is easy to see that $Q = (f_1,...,f_s)$, where
$$f_i = x_{i,1}\cdots x_{i,s-1},\ i = 1,...,s.$$ 
Hence, $Q$ is a complete intersection. From this it follows that $Q^{(t)} = Q^t$ for all $t \ge 1$. Therefore, 
$$I^{(t)} = P_1^t \cap \cdots \cap P_s^t \cap Q^t$$
by Lemma \ref{decomposition}.

Let $Q_i = P_i \cap Q$, $i = 1,...,s$. By Lemma \ref{decomposition}, $Q_i^{(t)} = P_i^t \cap Q^t$. Hence, 
$$I^{(t)} = Q_1^{(t)} \cap \cdots \cap Q_s^{(t)}.$$
In view of Lemma \ref{lem_depth_intersectionandsum}, we will estimate 
$\depth R/(Q_{j_1}^{(t)}+\cdots+Q_{j_i}^{(t)})$ for $1 \le i \le s$ and $1\le j_1< \cdots <j_i\le s$.
Without restriction we only need to estimate $\depth R/(Q_1^{(t)}+\cdots+Q_i^{(t)})$ for $1 \le i \le s.$
\medskip

\noindent {\bf Claim 1.} 
$\depth R/(Q_1^{(t)}+\cdots+Q_s^{(t)}) \ge 1$ if and only if $t \not\in E(s)$.

\begin{proof}[Proof of Claim 1]
We have
\[
Q_1^{(t)}+\cdots+Q_s^{(t)} = (P_1^t+\cdots+P_s^t) \cap Q^t.
\]
Since $P_1^t+\cdots+P_s^t$ is a primary ideal of the maximal homogeneous ideal of $R$,
$\depth R/(Q_1^{(t)}+\cdots+Q_s^{(t)}) \ge 1$ if and only if $Q^t \subseteq P_1^t+\cdots+P_s^t$.
It remains to show that $Q^t \subseteq P_1^t+\cdots+P_s^t$ if and only if $t \not\in E(s)$.

If $t\in E(s)$, we have $t= c(s-1)+r$, where $c \ge 1$, $1\le r \le s-1$ and $r \le c$. 
Set $g =(f_1\cdots f_{s-1})^cf_s^r$. Since $f_i \in Q$, $1\le i\le s$, we have $g \in Q^{c(s-1)+r}= Q^t$. 
For $i = 1,...,s$, let $d_i$ be the degree of $g$ in the variables $x_{i,1},\ldots,x_{i,s-1}$. 
Then $d_i=c(s-1)$ if $1\le i\le s-1$ and $d_i = r(s-1)$ if $i=s$. 
In both cases, we have $d_i \le c(s-1) < t$. 
Since $P_i =  (x_{i,1},\ldots,x_{i,s-1})$, this implies $g \notin P_i^t$ for $i = 1,...,s$. 
As $g$ is a monomial, $g \not\in P_1^t+\cdots+P_s^t$. 
Therefore, $Q^t \not\subseteq P_1^t+\cdots+P_s^t$.

If $t \not\in E(s)$, we have $t = c(s-1)+r$, where $1\le r \le s-1$ and $0 \le c \le r-1$.
Let $h = f_1^{c_1} \cdots f_s^{c_s}$ be an arbitrary minimal monomial generator of $Q^t = (f_1,...,f_s)^t$, 
where $c_1,\ldots,c_s\ge 0$ and $c_1+\cdots+c_s=t$.
Since $t > c(s-1)+c = cs$, at least one of the numbers $c_1,\ldots,c_s$ is greater than $c$. 
If $c_i > c$,  
$$c_i(s-1) \ge (c+1)(s-1) \ge c(s-1)+r = t.$$
By the definition of $f_i$ and $P_i$, we have $f_i \in P_i^{s-1}$.
Hence, $f_i^{c_i} \in P_i^{c_i(s-1)} \subseteq P_i^t$. So we get $h \in P_i^t \subseteq P_1^t+\cdots+P_s^t$.
Therefore, $Q^t \subseteq P_1^t+\cdots+P_s^t$, as desired.
\end{proof}

\noindent{\bf Claim 2.}  
$\depth R/(Q_1^{(t)}+\cdots+Q_i^{(t)}) \ge s-i+1$ for all $1 \le i \le s-1$.

\begin{proof}[Proof of Claim 2]
Note that
$$Q_1^{(t)}+\cdots+Q_i^{(t)} = (P_1^t+ \cdots+P_i^t) \cap Q^t.$$
Consider the exact sequence
\[
0 \to \frac{R}{(P_1^t+ \cdots+P_i^t) \cap Q^t} \to \frac{R}{P_1^t+ \cdots+P_i^t} \oplus  \frac{R}{Q^t} \to \frac{R}{P_1^t+ \cdots+P_i^t + Q^t} \to 0.
\]
Then we have
\begin{align*}
& \depth \frac{R}{Q_1^{(t)}+\cdots+Q_i^{(t)}} = \depth \frac{R}{(P_1^t+ \cdots+P_i^t) \cap Q^t}\\
& \ge \min  \left\{ \depth \frac{R}{P_1^t+ \cdots+P_i^t}, \depth \frac{R}{Q^t}, 
                                \depth \frac{R}{P_1^t+ \cdots+P_i^t + Q^t}+1 \right\}.
\end{align*}
Since $P_1,...,P_i$ are generated by disjoint sets of $s-1$ variables,
\[
\depth \frac{R}{P_1^t+ \cdots+P_i^t} = \depth \frac{R}{P_1+ \cdots+P_i} = (s-i)(s-1) \ge s-i+1.
\]
Since $Q$ is a complete intersection generated by $s$ elements, 
$$\depth R/Q^t=s(s-1)-s = s(s-2).$$
If $i = 1$, we have $s(s-2) \ge s+1$ because $s > 2$. If $i \ge 2$, we have $s(s-2) \ge s(i-1) \ge s \ge s - i + 1$.
Hence,
$$\depth R/Q^t \ge s-i+1.$$
It remains to show that
$$\depth \frac{R}{P_1^t+ \cdots+P_i^t + Q^t}  \ge s-i.$$
For this, it suffices to show that $x_{i+1,1}-x_{i+1,2},\ldots,x_{s,1}-x_{s,2}$ is a regular sequence for $R/(P_1^t+\cdots+P_i^t+Q^t)$.

It is clear that
$x_{i+1,1}-x_{i+1,2},\ldots,x_{s,1}-x_{s,2}$ form a regular sequence for all factor rings of the form $R/(x_{1,j_1},\ldots,x_{s,j_s})^t$. Note that
\[
Q^t = Q^{(t)} = \bigcap_{1\le j_1,\ldots,j_s \le s-1} (x_{1,j_1},\ldots,x_{s,j_s})^t
\]
by Lemma \ref{decomposition}. Then $x_{i+1,1}-x_{i+1,2},\ldots,x_{s,1}-x_{s,2}$ form a regular sequence for $R/Q^t$ by Lemma \ref{regular}(i). 
On the other hand, $P_1 + \cdots +P_i$ is the only associated prime of $P_1^t+\cdots+P_i^t$ and none of the minimal generators of 
$P_1 + \cdots +P_i$ involves any of the variables $x_{i+1,1},x_{i+1,2},\ldots,x_{s,1},x_{s,2}$.
Therefore, $x_{i+1,1}-x_{i+1,2},\ldots,x_{s,1}-x_{s,2}$ form a regular sequence for $R/(P_1^t+\cdots+P_i^t+Q^t)$ by Lemma \ref{regular}(ii). 
\end{proof}

Now we are going to use the above claims to prove that $\depth R/I^{(t)} = s-1$ if $t \in E(s)$ and 
$\depth R/I^{(t)} \ge s$ if $t \not\in E(s)$. \par

If $t\in E(s)$, then $\depth R/(Q_1^{(t)}+\cdots+Q_s^{(t)}) = 0$ by Claim 1. Together with Claim 2, we have
$$\depth R/(Q_1^{(t)}+\cdots+Q_i^{(t)}) \ge s-i$$
for all $1 \le i \le s$. By Lemma \ref{lem_depth_intersectionandsum}, this implies
\[
\depth R/I^{(t)}= \depth R/(Q_1^{(t)}\cap \cdots \cap Q_s^{(t)}) \ge s-1.
\]
By Corollary \ref{depth}, $\depth R/I^{(t)} \le s-1$ if there exists $\a \in \NN^{(s-1)s}$ such that 
$\widetilde{H}_{s-2}(\Delta_\a(I^{(t)}),k) \neq 0$.

By the definition of $E(s)$, $t= c(s-1)+r$, where $c \ge 1$, $1\le r \le s-1$, and $r \le c$. 
Let $\a \in \NN^{(s-1)s}$ such that $x^\a = (f_1\cdots f_{s-1})^cf_s^r$. 
Due to the proof of Claim 1, we have $x^\a \in Q^t \setminus (P_1^t+\cdots+P_s^t)$. 
Let $X$ denote the set of the variables of $R$.
For convenience, we consider $\D_\a(I^{(t)})$ as a simplicial complex on $X$.
By Proposition \ref{unmixed}, a subset $F \subseteq X$ is a  facet of $\Delta_\a(I^{(t)})$ if and only if $(X \setminus F)$ is a minimal prime of $I$ and $x^\a \not\in (X \setminus F)^t$.
Therefore,
\[
\Delta_\a(I^{(t)}) =\ <F_1,...,F_s>,
\]
where $F_i = X \setminus \{x_{i,1},\ldots,x_{i,s-1}\}$. Let 
$$\Gamma =\{A \subseteq [s]|\ \bigcap_{i\in A} F_i \neq \emptyset\},$$
which is the nerve of the complex $\Delta_\a(I^{(t)})$. 
By Borsuk's nerve theorem \cite[Theorem 10.7]{Bj},
$\Delta_\a(I^{(t)})$ and $\Gamma$ have the same homology groups. 
It is easy to see that $\Gamma$ is the boundary complex of the $(s-1)$-simplex $[1,\ldots,s]$. 
Therefore,
\[
\widetilde{H}_{s-2}(\Delta_\a(I^{(t)}),k) = \widetilde{H}_{s-2}(\Gamma,k) \neq 0.
\] 
So we get $\depth R/I^{(t)} = s-1$ if $t\in E(s)$. \par

If $t \not\in E(s)$, then $\depth R/(Q_1^{(t)}+\cdots+Q_s^{(t)}) \ge 1$ by Claim 1. Together with Claim 2, we have
$$\depth R/(Q_1^{(t)}+\cdots+Q_i^{(t)}) \ge s-i+1$$
for all $1 \le i \le s$. 
By Lemma \ref{lem_depth_intersectionandsum}, this implies
\[
\depth R/I^{(t)}=\depth R/(Q_1^{(t)}\cap \cdots \cap Q_s^{(t)}) \ge s.
\]
The proof of Theorem \ref{prop_example_increasing} is now complete.
\end{proof}

Though we have shown that the function $\depth R/I^{(t)}$ need not be non-increasing for a squarefree monomial $I$, we are not able to give an answer to the following question. 

\begin{quest} (cf. \cite[Problem 1.2]{KTY})
Is the function $\depth R/I^{(t)}$ non-increasing if $I$ is the edge ideal of a graph?
\end{quest}

Note that the analogous question for the function $\depth R/I^t$ is widely open (see e.g \cite{HH,HQ}).

%%%%%%%%%%%%%%%%%%%%%%%%%%%%%%%

\section{Convergent symbolic depth functions}

We call a numerical function $f(t)$ {\em asymptotically periodic} if $f(t)$ is a periodic function for $t \gg 0$.
This notion arises when we consider the symbolic depth functions of monomial ideals. \par

Let $R = k[x_1,...,x_n]$ be a polynomial ring over a field $k$. 

\begin{prop}  \label{Rees}
Let $I$ be a homogeneous ideal in $R$ whose symbolic Rees algebra is finitely generated.
Then $\depth R/I^{(t)}$ is an asymptotically periodic function.
\end{prop}

\begin{proof}
The finite generation of the symbolic Rees algebra $R_s(I) = \oplus_{t \ge 0}I^{(t)}$ implies that there is an integer $d \ge 1$ such that the Veronese subring $S := \oplus_{t \ge 0}I^{(dt)}$ is standard graded \cite[Theorem 2.1]{HHT}.  
For $i = 0,...,d-1$, set $M^{(i)} := \oplus_{t \ge 0}I^{(dt+i)}$. 
Then $M^{(i)}$ is a finitely generated graded module over $S$.
By \cite[Theorem 1.1]{HH}, $\depth I^{(dt+i)}$ is a constant for $t \gg 0$.
Hence, $\depth R/I^{(dt+i)} = \depth I^{(dt+i)}-1$ is a constant for $t \gg 0$.
Therefore, $\depth R/I^{(t)}$ is a periodic function with period $d$ for $t \gg 0$.
\end{proof}

\begin{cor}  \label{monomial}
Let $I$ be an arbitrary monomial ideal in $R$.
Then $\depth R/I^{(t)}$ is an asymptotically periodic function.
\end{cor}

\begin{proof}
By \cite[Theorem 3.2]{HHT}, 
the symbolic Rees algebra $R_s(I) = \oplus_{t \ge 0}I^{(t)}$ is finitely generated.
Hence the conclusion follows from Proposition \ref{Rees}.
\end{proof}

Corollary \ref{monomial} is more or less known among experts.
However, there have been no known examples of symbolic depth functions that are not convergent.  
It was an open question whether the symbolic depth function of a monomial ideal is always convergent \cite[p.~308]{HKTT}.

At first, we obtain the following positive result. Let
$$F_s(I) := \bigoplus_{t \ge 0}I^{(t)}/\mm I^{(t)},$$
where $\mm$ is the maximal homogeneous ideal of $R$.
Note that $F_s(I)$ is the fiber ring of the symbolic Rees algebra $R_s(I)$.

\begin{thm} \label{constant}
Let $I$ be a monomial ideal in $R$ such that $I^{(t)}$ is integrally closed for $t \gg 0$.
Then $\depth R/I^{(t)}$ is a convergent function with
$$\lim_{t \to\infty} \depth R/I^{(t)} = \dim R - \dim F_s(I),$$
which is also the minimum of $\depth R/I^{(t)}$ among all integrally closed symbolic powers $I^{(t)}$.
\end{thm}

\begin{proof}
Let $m$ be the minimum of $\depth R/I^{(t)}$ among all integrally closed symbolic powers $I^{(t)}$.
Choose an integrally closed symbolic power $I^{(s)}$ such that $\depth R/I^{(s)} = m$. 
By Theorem \ref{decreasing}(ii), there exists an integer $a$ such that $\depth R/I^{(s)}\ge \depth R/I^{(t)}$ for $t \ge  as^2$. This implies $\depth R/I^{(t)} = m$ for all integrally closed symbolic powers $I^{(t)}$ with $t \ge as^2$. Since $I^{(t)}$ is integrally closed for $t \gg 0$, we get $\depth R/I^{(t)} = m$ for $t \gg 0$. \par

We will show that $m = \min_{t \ge 1} \depth R/\overline{I^{(t)}}$.
Since $\overline{I^{(t)}} = I^{(t)}$ for $t \gg 0$, 
$$m \ge \min_{t \ge 1} \depth R/\overline{I^{(t)}}.$$
By Proposition \ref{integral}, we have
$$\depth R/\overline{I^{(s)}} \ge \depth R/\big(\overline{I^{(s)}}\big)^{(t)}$$ 
for all $s, t \ge 1$.
Using Lemma \ref{decomposition}, it is easy to check that 
$\overline{I^{(s)}} = \bigcap_{F \in \F(I)} \overline{I_F^s}.$
From this it follows that
$\big(\overline{I^{(s)}}\big)^{(t)} = \bigcap_{F \in \F(I)} \big(\overline{I_F^s}\big)^t.$
For $t \gg 0$, $I^{(st)}$ is integrally closed and so is $I_F^{st}$ for all $F \in \F(I)$ by Proposition \ref{closed}.
This implies 
$I_F^{st} \subseteq \big(\overline{I_F^s}\big)^t \subseteq \overline{I_F^{st}} = I_F^{st}.$
Hence, $\big(\overline{I_F^s}\big)^t = I_F^{st}$. So we get 
$$\big(\overline{I^{(s)}}\big)^{(t)} = \bigcap_{F \in \F(I)} I_F^{st} = I^{(st)}.$$
Therefore,
$$\depth R/\overline{I^{(s)}} \ge \depth R/ I^{(st)} \ge m$$
for all $s \ge 1$. Now, we can conclude that
$$m = \min_{t \ge 1} \depth R/\overline{I^{(t)}}.$$

It remains to show that $\min_{t \ge 1} \depth R/\overline{I^{(t)}} = \dim R - \dim F_s(I).$ For that we need the following auxiliary observation (cf. \cite[Proposition 2.5]{Va}).\par

Let $\F$ denote the filtration of the ideals $\overline{I^{(t)}}$, $t \ge 0$.
Let $R(\F) =  \bigoplus_{t \ge 0}\overline{I^{(t)}}y^t$.
Then $R(\F)$ is an algebra generated by monomials in $k[x_1,...,x_n,y]$.
Since $\overline{I^{(t)}} = \bigcap_{F \in \F(I)} \overline{I_F^t}$, we have
$$R(\F) =  \bigcap_{F \in \F(I)} \bigoplus_{t \ge 0}\overline{I_F^{t}}y^t.$$
For each $F \in \F(I)$, the algebra $\bigoplus_{t \ge 0}\overline{I_F^{t}}y^t$ is the normalization of the finitely generated algebra $\bigoplus_{t \ge 0}I_F^{t}y^t$.
Hence, $\bigoplus_{t \ge 0}\overline{I_F^{t}}y^t$ is a finitely generated algebra. 
The monomials of $\bigoplus_{t \ge 0}\overline{I_F^{t}}y^t$ form a finitely generated semigroup. 
Since the semigroup of the monomials of $R(\F)$ is the intersections of these semigroups, it is also finitely generated \cite[Corollary 1.2]{HHT}. 
From this it follows that $R(\F)$ is a finitely generated algebra. 
Moreover, as an intersection of normal rings, $R(\F)$ is a normal ring. 
By \cite[Theorem 1]{Ho}, this implies that $R(\F)$ is Cohen-Macaulay.
\par

Let $G(\F) =  \bigoplus_{t \ge 0}\overline{I^{(t)}}/\overline{I^{(t+1)}}$. 
Then $G(\F)$ is a factor ring of $R(\F)$ by the ideal  $\bigoplus_{t \ge 0}\overline{I^{(t+1)}}y^t$.
Hence, $G(\F)$ is a finitely generated algebra. By \cite[Theorem 4.5.6(b)]{BH}, we have $\dim G(\F) = \dim R$.
By the proof of the necessary part of \cite[Theorem 1.1]{TI},  
the Cohen-Macaulayness of $R(\F)$ implies that of  $G(\F)$. 
By \cite[Theorem 9.23]{BV}, these facts imply
\begin{align*}
\min_{t \ge 1} \depth R/\overline{I^{(t)}} & = \grade \mm G(\F) = \height \mm G(\F) \\
& =  \dim G(\F) - \dim G(\F)/\mm G(\F)\\
& = \dim R - \dim G(\F)/\mm G(\F).
\end{align*}
We have $G(\F)/\mm G(\F) = \bigoplus_{t \ge 0}\overline{I^{(t)}}/\mm\overline{I^{(t+1)}}$.
Since $F_s(I) = \bigoplus_{t \ge 0} I^{(t)}/\mm I^{(t+1)}$ and $\overline{I^{(t)}} = I^{(t)}$ for $t \gg 0$, the graded algebras $G(\F)/\mm G(\F)$ and $F_s(I)$ share the same Hilbert quasi-polynomial \cite[Theorem 4.4.3]{BH}. From this it follows that
$\dim G(\F)/\mm G(\F) = \dim F_s(I)$. Therefore,
$$\min_{t \ge 1} \depth R/\overline{I^{(t)}} = \dim R -\dim F_s(I).$$
\end{proof}

The convergence of the function $\depth R/I^{(t)}$ in Theorem \ref{constant} can be also deduced from \cite[Theorem 4.7]{HT1}, which estimates the Castelnuovo-Mumford regularity of integral closures of ideals of the form $I_1^t \cap \cdots \cap I_p^t$, where $I_1,...,I_p$ are monomial ideals. 
However, our proof gives more information on $\lim_{t\to\infty} \depth R/I^{(t)}$ and it is more transparent.

\begin{cor} \label{normal}  
Let $I$ be a monomial ideal in $R$ such that $I^{(t)}$ is integrally closed for all $t \ge 1$.
Then $\depth R/I^{(t)}$ is a convergent function with
$$\lim_{t \to\infty} \depth R/I^{(t)} = \min_{t \ge 1} \depth R/I^{(t)} = \dim R - \dim F_s(I).$$
\end{cor}

\begin{proof}
This is straightforward from Theorem \ref{constant}.
\end{proof}

We have seen in the proof of Corollary \ref{squarefree} that $I^{(t)}$ is integrally closed for all $t \ge 1$ if $I$ a squarefree monomial ideal $I$. If we apply Corollary \ref{normal} to this case, we obtain the main part of \cite[Theorem 2.4]{HKTT}. 
\par

Now, we will present a large class of monomial ideals $I$ with the property that $I^{(t)}$ are integrally closed for $t \gg 0$.
This follows from the following observation.

\begin{lem} \label{same degree}
Let $Q$ be a primary ideal in $R$ generated by monomials of the same degree $d$.
Let $F \subseteq [n]$ such that $P_F$ is the associated prime of $Q$.
Then $Q^t$ is integrally closed for $t \gg 0$ if and only if $Q$ contains the monomials $x_i^{d-1}x_j$ for all $i, j \not\in F$.
\end{lem}

\begin{proof}
We may assume that $d \ge 2$.
Since $Q$ is a $P_F$-primary ideal generated by forms of degree $d$, $x_i^d \in Q$ for all $i \not\in F$. 
This implies $x_i^{dt} \in Q^t$ for $t \ge 1$.
Hence, the convex polyhedron $N(Q^t) := \conv\{\a \in \NN^n\mid x^\a  \in Q^t\}$ contains all integral points $\a \in \NN^n$ with $|\a| = dt$. 
If $Q^t$ is integrally closed, $x^\a \in Q^t$ for all such $\a$ (see e.g. \cite[\textsection 1.4]{HS}).
In particular, $x_i^{td-1}x_j \in Q^t$ for all $i, j \not\in F$. 
This condition is satisfied only if $Q$ contain $x_i^{d-1}x_j$. \par

If $Q$ contains the monomials $x_i^{d-1}x_j$ for all $i, j \not\in F$, 
we will show that $Q^t = P_F^{dt}$ for $t \ge n(d-1)$. 
Note that $P_F^{dt}$ is generated by the monomials of degree $td$ in the variables $x_i$, $i \not\in F$.
Let $x^\a$ be such a monomial. Note that $|\a| = dt$ and $a_i = 0$ for all $i \in F$.
We will show that there exist $\b, \c \in \NN^n$ such that $\a = (d-1)\b + \c$ with $|\b| = |\c|$. 
 \par

We will start from some expressions $\a = (d-1)\b + \c$ such that $|\b| \ge |\c|$.
For instance, we can always write $a_i = (d-1)b_i + c_i$ with $0 \le c_i < d-1$.
Since $|\c| \le n(d-2)$ and $t \ge n(d-1)$, we have
\begin{align*}
|\b| & = (|\a| - |\c|)/(d-1) \ge (td - n(d-2))/(d-1) \\
& \ge (n(d-2)d - n(d-2))/(d-1) = n(d-2) \ge |\c|.
\end{align*}
It suffices to show that if $|\b| > |\c|$, we can find other expressions $a_i = b_i'(d-1) + c_i'$ with $b_i', c_i' \ge 0$ such that $|\b'| \ge |\c'|$ and $|\b'|-|\c'| < |\b|-|\c|$. \par

Note that $|\b|-|\c| = |\a| - d|\b| = d(t - |\b|)$. If $|\b| > |\c|$, then $t - |\b| > 0$. Hence, $|\b| \ge |\c| + d$. 
Choose $j$ such that $b_j > 0$. Set $b_j' = b_j - 1$, $c'_j = c_j + d-1$, and $b_i' = b_i$, $c_i' = c_i$ for $i \neq j$.
Then $a_i = (d-1)b_i' + c_i'$ for all $i$. Hence, $\a = (d-1)\b' + \c'$. Since $|\b'| = |\b| - 1$ and $\c'  = |\c| + d-1$, we also have $|\b'| \ge |\c'|$ and $|\b'|-|\c'| < |\b| - |\c|$.

Once we have found $\b, \c \in \NN^n$ such that $\a = (d-1)\b + \c$ with $|\b| = |\c|$, we rewrite $(d-1)\b + \c$ as a sum of $|\b|$ vectors of the form $(d-1)\e_i + \e_j$. Since $|\a| = dt = (d-1)|\b| +|\c|$, we must have $|\b| = t$. Therefore, we can write $x^\a$ as a product of $t$ monomials of the form $x_i^{d-1}x_j$. Since $a_i = 0$ for $i \in F$, we also have $b_i = c_i = 0$ for $i \in F$. From this it follows that these monomials $x_i^{d-1}x_j$ belong to $Q$, which implies $x^\a \in Q^t$. 
So we can conclude that $Q^t = P_F^{dt}$, which is integrally closed.
\end{proof}

\begin{prop} \label{primary}
Let $I$ be a monomial ideal such that for all $F \in \F(I)$, the primary component $I_F$ is generated by forms of the same degree. Then $I^{(t)}$ is integrally closed for $t \gg 0$ if and only if each ideal $I_F$ contains the monomials $x_i^{d-1}x_j$ for all $i, j \not\in F$, where $d$ is the generating degree of $I_F$.
\end{prop}

\begin{proof}
By Proposition \ref{closed}, $I^{(t)}$ is integrally closed if and only if $I_F^t$ is integrally closed for all $F \in \F(I)$.
By Lemma \ref{same degree}, $I_F^t$ is integrally closed for $t \gg 0$ if and only if $I_F$ contain the monomials $x_i^{d-1}x_j$ for all $i, j \not\in F$.
\end{proof}

%%%%%%%%%%%%%%%%%%%%%%%%%%%%%%%%%

\section{Basic symbolic depth functions}

The aim of this section is to exhibit symbolic depth functions which will be used to build up arbitrary
asymptotically periodic symbolic depth functions. 
\medskip

First, we present a class of monomial ideals $I$ with $\depth R/I^{(t)} = 1, 2$ for all $t \ge 1$, for which we are able to check when $\depth R/I^{(t)} = 1$ or $\depth R/I^{(t)} = 2$.

\begin{prop}
\label{lem_key}
Let $R=k[x,y,z,u,v]$ be a polynomial ring. Let $M,P,Q$ be primary monomial ideals of $R$ such that 
\begin{align*}
\sqrt{M} & =(x,y,z),\\
\sqrt{P} & =(x,y),\\
\sqrt{Q} & =(z).
\end{align*}
Let $I=M\cap (P,u) \cap (Q,v)$.  For all $t \ge 1$, \par
{\rm (i)} $\depth R/I^{(t)} \le 2$,\par
{\rm (ii)} $\depth R/I^{(t)}=2$ if and only if $M^t \subseteq P^t + Q^t.$
\end{prop}

\begin{proof}
(i)  We always have $\depth R/I^{(t)} \ge 1$.
Hence, it suffices to show that $\depth R/I^{(t)} \le  2$. 
By Lemma \ref{primary} we have the primary decomposition
$$I^{(t)}=M^t \cap (P,u)^t \cap (Q,v)^t.$$
From this it follows that $\depth R/I^{(t)} \le \dim R/M^t = 2$ \cite[Proposition 1.2.13]{BH}. \par

(ii)  It suffices to show that $\depth R/I^{(t)} \ge 2$ if and only if $M^t \subseteq P^t+Q^t$.
It is easy to see that $I^{(t)}$ satisfies the conditions (i) and (ii) of Proposition \ref{depth2}.
Therefore, we only need to show that for $\a \in \NN^5$, 
the degree complex $\D_\a(I^{(t)})$ is connected if and only if $M^t \subseteq P^t+Q^t$. 
Note that $G_\a = \emptyset$ and $\a^+ = \a$ for $\a \in \NN^5$. 
\par

For convenience, we consider the sets of $\F(I^{(t)})$ and $\D_\a(I^{(t)})$ as sets of variables of $R$.
Then
$$\F(I^{(t)}) =  \big\{\{u,v\},\{z,v\},\{x,y,u\}\big\}.$$
By Proposition \ref{unmixed}, each facet of $\D_\a(I^{(t)})$ belongs to $\F(I^{(t)})$.
Therefore, $\D_\a(I^{(t)})$ is disconnected if and only if 
$$
\D_\a(I^{(t)}) = \left<\{z,v\}, \{x,y,u\}\right>.
$$

Let $x^\a$ denote the monomial of $R$ having the exponent vector $\a$ (the order of the variables is $x,y,z,u,v$).
By Proposition \ref{unmixed}, $\{u,v\}$ or $\{z,v\}$ or $\{x,y,u\}$ is a facet of
$\D_\a(I^{(t)})$ if and only if $x^\a \not\in M^t$ or $x^\a \not\in (P,u)^t$ or $x^\a \not\in (Q,v)^t$, respectively. 
Therefore, $\D_\a(I^{(t)}) = \left<\{z,v\}, \{x,y,u\}\right>$ if and only if $x^\a \in M^t$, $x^\a \not\in (P,u)^t$, $x^\a \not\in (Q,u)^t$.
This condition is satisfied for some $\a \in \NN^5$ if and only if there exists a monomial generator of $M^t$ which is not contained in $(P,u)^t \cup (Q,v)^t$.
Since $u,v$ do not appear in the minimal monomial generators of $M^t$, this is equivalent to the condition
that there exists a minimal monomial generator of $M^t$ which is 
not contained in $P^t \cup Q^t$.
As we are dealing with monomial ideals, the last condition is satisfied 
 if and only if $M^t \not\subseteq  P^t + Q^t$.
So we can conclude that $\depth R/I^{(t)} \ge 2$ if and only if $M^t \subseteq P^t + Q^t$.
\end{proof}

It is clear that $M^t \not\subseteq P^t + Q^t$ if and only if every monomial generator of $M^t$ is not divisible by any monomial generator of $P^t$ and $Q^t$. To check this condition amounts to solving a system of linear diophantine inequalities. The solvability of such a system depends very much on $t$. Therefore, we can use
Proposition \ref{lem_key} to construct symbolic depth functions with irregular behavior. \par

For convenience, we identify a numerical function $\phi(t)$ with the sequence 
$$\phi(1), \phi(2),\phi(3), ...\ .$$
Now we are going to construct symbolic depth functions of the following types.
\par

\begin{enumerate}
\item[{\bf  A}:] 1,...,1,2,2,... , which is a monotone function converging to 2,
\item[{\bf  B}:] 1,...,1,2,1,1,... , which has the value 2 at only one position, 
\item[{\bf  C}:] 1,1,1,...  or 1,..,1,2,1,..,1,1,..,1,2,1,..,1,... , which is a periodic function with a period of the form 1,..,1,2,1,..,1, where 2 can be at any position. 
\end{enumerate}

Note that type {\bf A} shows the existence of monomial ideals whose symbolic depth function is convergent with 
$$\lim_{t \to\infty} \depth R/I^{(t)} \neq \min_{t \ge 1} \depth R/I^{(t)},$$ 
that is unlike the formula for squarefree monomial ideals in Corollary \ref{normal}. 
Type {\bf C} shows the existence of monomial ideals whose symbolic depth function is not convergent. \par

For type {\bf A} we have the following class of ideals.

\begin{lem} \label{type1}
Let $m\ge 2$ be an integer. Let $R=k[x,y,z,u,v]$ and 
\[
I=(x^{2m-2},y^m,z^{2m})^2 \cap (x^{2m-1},y^{2m-1},u) \cap (z,v).  
\]
Then 
$$
\depth R/I^{(t)} = 
\begin{cases}
1 &\text{if $t\le m-1$},\\
2 &\text{if $t\ge m$}.
\end{cases}
$$
\end{lem}

\begin{proof}
Let
\begin{align*}
M &=(x^{2m-2},y^m,z^{2m})^2,\\
P &=(x^{2m-1},y^{2m-1}),\\
Q &=(z).
\end{align*}
By Proposition \ref{lem_key}, we have to show that $M^t \subseteq P^t+Q^t$ if and only if $t\ge m$.

A minimal monomial generator of $M^t=(x^{2m-2},y^m,z^{2m})^{2t}$ has the form
\[
f = (x^{(2m-2)})^{\ell}(y^m)^{i-\ell}(z^{2m})^{2t-i}=x^{(2m-2)\ell}y^{m(i-\ell)}z^{2m(2t-i)}
\]
where $0\le \ell \le i\le 2t$. Note that $f \in P^t+Q^t$ if and only if $f \in P^t$ or $f \in Q^t$.
Comparing the exponents of $x,y,z$ in $f$ with those in the minimal monomial generators of $P^t+Q^t$, we see that $M^t \subseteq P^t+Q^t$ if and only if the following system has no solution:
\begin{align}
0 &\le \ell  \le i \le 2t, \label{eq_elli}\\
\left\lfloor \frac{(2m-2)\ell}{2m-1} \right \rfloor &+\left\lfloor \frac{m(i-\ell)}{2m-1} \right \rfloor < t, \label{eq_ineqxy}\\
2m(2t-i) &<t. \label{eq_ineqz}
\end{align} 

\textbf{Case 1}: $t\le m-1$. \par
\smallskip

The system has the solution $\ell=1, i=2t$. Indeed, it suffices to check whether
\[
\left\lfloor \frac{2m-2}{2m-1} \right \rfloor +\left\lfloor \frac{m(2t-1)}{2m-1} \right \rfloor < t.
\]
This is true since
\[
\left\lfloor \frac{2m-2}{2m-1} \right \rfloor=0,\;
\left\lfloor \frac{m(2t-1)}{2m-1} \right \rfloor \le  \left\lfloor \frac{2mt - t-1}{2m-1} \right \rfloor  =  t-1.
\]

\textbf{Case 2}: $t\ge m$. \par
\smallskip

We show that the system has no solution. 
Note that the sum of the two fractions inside the integral parts of \eqref{eq_ineqxy} must be less than $t+1$. Then
\[
(2m-2)\ell+m(i-\ell) < (2m-1)(t+1),
\]
which implies
\begin{equation}
\label{eq_ineq_elli}
(m-2)\ell  \le m(2t-i)+2m-t-2.
\end{equation}

If $i=2t$, we have 
$$(m-2)\ell\le 2m-t-2 \le m-2.$$
Hence $\ell \le 1$. But then \eqref{eq_ineqxy} does not hold because
\[
\left\lfloor \frac{m(i-\ell)}{2m-1} \right \rfloor  \ge \left\lfloor \frac{m(2t-1)}{2m-1} \right \rfloor \ge t.
\]

If $i\le 2t-1$,  \eqref{eq_ineqz} implies $2m \le t-1$. 
Combining \eqref{eq_ineqz} with \eqref{eq_ineq_elli}, we get
\[
(m-2)\ell \le \frac{t-1}{2}+2m-t-2=2m-3-\frac{t-1}{2} \le m-3.
\]
From this it follows that $\ell=0$ and 
$$0 \le 2m-3- \frac{t-1}{2}.$$
Hence, $t \le 4m-5.$ Now, \eqref{eq_ineqz} implies $2m(2t - i) \le 4m-5$.
Therefore, $2t-i \le 1$. 
So we get $i=2t-1$. But then \eqref{eq_ineqxy} does not hold because
\[
\left\lfloor \frac{m(i-\ell)}{2m-1} \right \rfloor  \ge \left\lfloor \frac{m(2t-1)}{2m-1} \right \rfloor \ge t.
\]

The proof of Lemma \ref{type1} is now complete.
\end{proof}

The existence of symbolic depth functions of type {\bf B} is established with the following lemma.

\begin{lem}
\label{type2}
Let $m\ge 1$ be an integer. Let $R=k[x,y,z,u,v]$ and 
\[
I=(x^{2m},y^{2m},xy^{m-1}z,z^{2m})^2 \cap (x^m,y^m,u) \cap (z^{2m+2},v).  
\]
Then 
$$
\depth R/I^{(t)}= \begin{cases}
2 &\text{if $t = m$},\\
1 &\text{if $t \neq m$}.
\end{cases}
$$
\end{lem}

\begin{proof}
Let 
\begin{align*}
M &=(x^{2m},y^{2m},xy^{m-1}z,z^{2m})^2,\\
P &=(x^m,y^m),\\
Q &=(z^{2m+2}).
\end{align*}
By Proposition \ref{lem_key}, we have to show that $M^t \subseteq P^t+Q^t$ if and only if $t =m$. 
Note that
$$M^t = \sum_{i,j \ge 0,\ i+j \le 2t} (x^{2m},y^{2m})^i(xy^{m-1}z)^jz^{2m(2t-i-j)}.$$

{\bf Case 1}: $t = m$. \par
\smallskip

Let $J = (x^{2m},y^{2m})^i(xy^{m-1}z)^jz^{2m(2m-i-j)}$ for some $i,j \ge 0$, $i+j \le 2m.$
We will show that $J \subseteq P^m+Q^m$. 
Assume to the contrary that $J \not\subseteq P^m + Q^m$. Then $J \not\subseteq P^m$ and $J \not\subseteq Q^m$.

Write 
\begin{align*} 
j  & = pm + a,\ 0 \le a < m,\\
(m-1)j & = qm + b,\ 0 \le b < m.
\end{align*}
Then $J \subseteq (x^{2m},y^{2m})^ix^{pm}y^{qm}$. Hence, $(x^{2m},y^{2m})^ix^{pm}y^{qm} \not\subseteq P^m$. Let $f$ be a minimal monomial generator of $(x^{2m},y^{2m})^ix^{pm}y^{qm}$ which is not contained in $P^m$.
Since $f$ is a products of $x^m,y^m$  and since $P^m = (x^m,y^m)^m$, the degree of $f$ must be less than the degree of the minimal monomial generators of $P^m$. Hence, $2mi + pm + qm < m^2.$
We have 
\begin{align*}
2mi + pm + qm - m^2 & = 2mi + mj - a - b - m^2\\
& = m(2i + j - m) -  (a+b).
\end{align*}
Therefore, $m(2i + j - m) < a+b$. Note that $a+b < 2m$ and that
$a+b = jm - pm - qm$ is divisible by $m$. Then $a + b \le m$. This implies $m(2i + j - m) < m$.
Hence, $2i + j - m \le 0$ or $2i+j \le m$. \par

On the other hand, we have $J \subseteq (z^{2m(2m-i-j)})$. Hence, $z^{2m(2m-i-j)} \not\in Q^m = (z^{(2m+2)m})$.
This implies $2m(2m-i-j) < (2m+2)m.$
Hence, $2m - i - j < m+1$ or $m < i+j +1$. 

Summing up, we have $2i+j \le m \le i+j,$
which implies $i = 0$ and $j = m$.  
But then we have $J = (x^my^{(m-1)m}z^{m(2m+1)}) \subseteq P^m$, a contradiction. 
\smallskip

{\bf Case 2}:  $t$ is not divisible by $m$. \par
\smallskip

Consider the monomial $(xy^{m-1})^tz^{(2m+1)t} = (xy^{m-1}z)^t(z^{2m})^t \in M^t$. 
Since $t$ is not divisible by $m$, $(xy^{m-1})^t$ is not a product of the monomials $x^m, y^m$.
Note that $(xy^{m-1})^t$ and the minimal generators of $P^t = (x^m,y^m)^t$ have the same degree.
Then $(xy^{m-1})^t \not\in P^t$.  
Hence, $(xy^{m-1})^tz^{(2m+1)t} \not\in P^t$.
Since $(2m +1)t < (2m+2)t$, we have $z^{(2m+1)t}  \not\in (z^{(2m+2)t}) = Q^t$. 
Hence, $(xy^{m-1})^tz^{(2m+1)t} \not\in Q^t$.
Therefore, $(xy^{m-1})^tz^{(2m+1)t} \not\in P^t+Q^t$.
This implies $M^t \not\subseteq P^t+Q^t$. \par
\smallskip

{\bf Case 3}: $t \neq m$ and $t$ is divisible by $m$.
\smallskip

Consider the monomial $(xy^{m-1})^{t-1}z^{(2m+1)t+2m-1} = (xy^{m-1}z)^{t-1}(z^{2m})^{t+1} \in M^t$.
We have $\deg (xy^{m-1})^{t-1} = m(t-1) < mt$. 
Since $mt$ is the degree of the minimal generators of $P^t = (x^m,y^m)^t$, this implies 
$(xy^{m-1})^{t-1} \not\in P^t$. Hence, $(xy^{m-1})^{t-1}z^{(2m+1)t+2m-1} \not\in P^t$.
Since $t \neq m$ and $t$ is divisible by $m$, we have $2m \le t$. Hence,
$(2m+1)t+2m -1 < (2m+2)t$. 
This implies $z^{(2m+1)t+2m -1} \not\in (z^{(2m+2)t}) = Q^t$. 
Hence, $(xy^{m-1})^{t-1}z^{(2m+1)t+2m -1} \not\in Q^t$. 
Therefore, $(xy^{m-1})^{t-1}z^{(2m+1)t+2m -1} \not\in P^t+Q^t$. 
This shows that $M^t \not\subseteq P^t+Q^t$.
\par

The proof of Lemma \ref{type2} is now complete.
\end{proof}

For type {\bf C} we first note that the existence of the symbolic depth function $1,1,1,...$ is trivial, e.g. with
$R = k[x,y]$ and $I = (x)$. 
The existence of other symbolic depth functions of type {\bf C} follows from the following result.

\begin{thm}
\label{periodic}
Let $m\ge 2$ and $0 \le d < m$ be integers. 
There exists a monomial ideal $I$ in $R=k[x,y,z,u,v]$ such that
$$
\depth R/I^{(t)}= \begin{cases}
2 &\text{if $t \equiv d$ modulo $m$},\\
1 &\text{otherwise}.
\end{cases}
$$
\end{thm}

The construction of the ideal $I$ in the proof of Theorem \ref{periodic} depends 
on $d$. For $d = 0$ we have the following ideals.

\begin{lem}
\label{thm_periodic_d=m}
Let $m\ge 2$ be an integer. Let 
$R=k[x,y,z,u,v]$ and 
\[
I= (x^{2m-2},y^m,z^m)^2 \cap (x^{2m-1},y^{2m-1},u) \cap (z,v). 
\]
Then 
$$
\depth R/I^{(t)}= \begin{cases}
2 &\text{if $t \equiv 0$ modulo $m$},\\
1 &\text{otherwise}.
\end{cases}
$$
\end{lem}

\begin{proof}  
Let 
\begin{align*}
M &=(x^{2m-2},y^m,z^m)^2,\\
P &=(x^{2m-1},y^{2m-1}),\\
Q &=(z).
\end{align*}
By Proposition \ref{lem_key}, it suffices to show that $M^t \subseteq P^t+Q^t$ if and only if $t$ is divisible by $m$.
\par

A minimal monomial generator of $M^t$ has the form
\[
f = x^{(2m-2)i}y^{mj}z^{m(2t-i-j)}
\]
where $i,j\ge 0, i+j\le 2t$. Note that $f \in P^t+Q^t$ if and only if $f \in P^t$ or $f \in Q^t$.
Comparing the exponents of $x,y,z$ of $f$  with those of generators of $P^t$ and $Q^t$, we see that $M^t \subseteq P^t+Q^t$ if and only if the following system has no solution:
\begin{align}
i,j \ge 0, i+j &\le 2t, \label{eq_ij}\\
\biggl\lfloor \frac{(2m-2)i}{2m-1} \biggr\rfloor &+ \biggl \lfloor \frac{mj}{2m-1} \biggr\rfloor <t, \label{eq_power_xy}\\
m(2t-i-j) &<t. \label{eq_power_z}
\end{align}

{\bf Case 1}: $t$ is not divisible by $m$.
\smallskip

Set $\alpha= \lfloor(t-1)/m\rfloor$. Let $i=1$ and $j=2t-1-\alpha$. 
Then \eqref{eq_ij} is satisfied. For \eqref{eq_power_xy}, we have
\[
\biggl\lfloor \frac{(2m-2)i}{2m-1} \biggr\rfloor =\biggl\lfloor \frac{2m-2}{2m-1} \biggr\rfloor=0.
\]
Note that $\alpha +1 > (t-1)/m$. Then $t < m(\alpha+1)+1$. 
Since $t$ is not divisible by $m$, we have $t \le m(\alpha+1)-1$.
Therefore,
\[
\biggl\lfloor \frac{mj}{2m-1} \biggr\rfloor = \biggl\lfloor  \frac{m(2t-1-\alpha)}{2m-1} \biggr\rfloor  \le \biggl\lfloor \frac{2mt-t-1}{2m-1}\biggr\rfloor < t.
\]
Hence, \eqref{eq_power_xy} is satisfied. For \eqref{eq_power_z}, we have
\[
m(2t-i-j)=m\alpha \le t-1 < t.
\]
Therefore, the system \eqref{eq_ij}--\eqref{eq_power_z} has a solution in this case.
\smallskip

{\bf Case 2}: $t$ is divisible by $m$.
\smallskip

Assume that the above system has a solution $(i,j)$. Then
the sum of the two terms inside the integral parts of \eqref{eq_power_xy}  are less than $t+1$. Hence
\begin{equation}
\label{eq_power_xy_1}
(2m-2)i+mj \le (2m-1)(t+1)-1.
\end{equation}
Since $t$ is divisible by $m$, \eqref{eq_power_z} implies
\begin{equation}
\label{eq_power_z_1}
m(2t-i-j) \le t-m.
\end{equation}

If $m=2$,  \eqref{eq_power_xy_1} becomes
\[
2i+2j \le 3(t+1)-1=3t+2.
\]
From \eqref{eq_power_z_1} we get $3t+2 \le 2i+2j$. 
Hence, $2i+2j=3t+2$. By \eqref{eq_power_xy} we have
\[
\biggl\lfloor \frac{2i}{3} \biggr\rfloor + \biggl \lfloor \frac{2j}{3} \biggr\rfloor <t.
\]
Since $2j=3t+2-2i$, we obtain
\begin{equation}
\label{eq_power_xy_2}
\biggl\lfloor \frac{2i}{3} \biggr\rfloor + \biggl \lfloor \frac{2-2i}{3} \biggr\rfloor <0.
\end{equation}
Write $2i=3u+v$, where $0\le v\le 2$. Then
\begin{align*}
\biggl\lfloor \frac{2i}{3} \biggr\rfloor  =u\ \text{and}\ \biggl\lfloor \frac{2-2i}{3} \biggr\rfloor  = \biggl\lfloor \frac{2-v}{3} \biggr\rfloor -u =-u,
\end{align*}
which gives a contradiction to \eqref{eq_power_xy_2}.

If $m\ge 3$, \eqref{eq_power_z_1} implies $mj \ge m(2t-i)-t+m.$  Combining it with \eqref{eq_power_xy_1} we obtain
\[
(2m-2)i+m(2t-i)-t+m \le (2m-1)(t+1)-1
\]
or equivalently, $(m-2)i \le m-2$. Thus, $i \le 1$. 
Now, \eqref{eq_power_z_1} implies  $mj\ge (2m-1)t$. But then 
\[
\biggl\lfloor \frac{mj}{2m-1} \biggr\rfloor \ge t,
\]
contradicting \eqref{eq_power_xy}.

Therefore, the system \eqref{eq_ij}--\eqref{eq_power_z} has no solution in this case.
The proof of Lemma \ref{thm_periodic_d=m} is now complete.
\end{proof}

For $d >  0$, the proof of Theorem \ref{periodic} is based on the following construction.

\begin{lem}
\label{thm_periodic_large_d}
Let $m\ge 4$ and $d\ge 2$ be integers such that $\sqrt{m/2} \le d \le m/2$. Let 
\begin{align*}
M &=(x^{2m+1-d},y^{m+2d-1},z^{m+2d-1},xy^{m-1}z)^2,\\
P &=(x^{2m},y^{2m}),\\
Q &=(z^2).
\end{align*} 
Then $M^t \subseteq P^t +Q^t$  if and only if $t\equiv d$ modulo $m$.
\end{lem}

\begin{proof}
A minimal monomial generator of $M^t$ has the form
\begin{align*}
f & = (x^{2m+1-d})^{\ell}(y^{m+2d-1})^{i-\ell}(z^{m+2d-1})^j(xy^{m-1}z)^{2t-i-j}\\
& = x^{(2m+1-d)\ell + 2t-j-i}y^{(m+2d-1)(i-\ell)+(m-1)(2t-i-j)}z^{(m+2d-1)j+2t-i-j},
\end{align*}
where 
\begin{align}
0  & \le  \ell \le i \le 2t, \label{eq_ineq10}\\
0   & \le j  \le 2t-i. \label{eq_ineq11}
\end{align}
Note that $f \in P^t +Q^t$  if and only if $f \in P^t$ or $f \in Q^t$.
It is clear that $f \not\in P^t$ if and only if 
$$x^{(2m+1-d)\ell + (2t-j-i)}y^{(m+2d-1)(i-\ell)+(m-1)(2t-i-j)} \not\in (x^{2m},y^{2m})^t$$
if and only if
\begin{equation*}
\left\lfloor \frac{(2m+1-d)\ell + 2t-j-i}{2m} \right \rfloor + \left\lfloor \frac{(m+2d-1)(i-\ell)+(m-1)(2t-i-j)}{2m}  \right\rfloor < t.
\end{equation*}
Write $t=qm+e$, where $0 \le e \le m-1$. We can express this condition as  
\begin{equation}
\left\lfloor \frac{(2m-d+1)\ell+2e-i-j}{2m} \right \rfloor + \left\lfloor \frac{2di-2e-(m+2d-1)\ell-(m-1)j}{2m}  \right\rfloor < 0.
\label{eq_ineq3}
\end{equation}
We have $f \not\in Q^t$ if and only if $z^{(m+2d-1)j+2t-i-j} \not\in (z^{2t})$ if and only if
\begin{equation}
(m+2d-2)j - i < 0. \label{eq_ineq2} 
\end{equation}
Therefore, $M^t \subseteq P^t +Q^t$ if and only if the system of inequalities \eqref{eq_ineq10}--\eqref{eq_ineq2} has no solution. It suffices now to show that this system has no solution if and only if $e = d$.
\smallskip

\textbf{Case 1}:  $e = 0$ or $d < e \le m-1$.
\smallskip

Choose $i=1,j=0,\ell=0$. Clearly, \eqref{eq_ineq10}, \eqref{eq_ineq11}, \eqref{eq_ineq2} are fulfilled. Now \eqref{eq_ineq3} becomes
\[
\left\lfloor \frac{2e-1}{2m} \right \rfloor + \left\lfloor \frac{2d-2e}{2m}  \right\rfloor < 0,
\]
which is true since for $d+1\le e \le m-1$,
$$
\left\lfloor \frac{2e-1}{2m} \right \rfloor =0,\
\left\lfloor \frac{2d-2e}{2m}  \right\rfloor =-1,
$$
and for $e=0$, 
$$
\left\lfloor \frac{2e-1}{2m} \right \rfloor =-1,\ 
\left\lfloor \frac{2d-2e}{2m}  \right\rfloor =0.
$$

\textbf{Case 2}: $e=1$.
\smallskip

Choose $i=2,\ell=1,j=0$. Again \eqref{eq_ineq10}, \eqref{eq_ineq11},  \eqref{eq_ineq2} are fulfilled, while \eqref{eq_ineq3} becomes
\[
\left\lfloor \frac{2m-d+1}{2m} \right \rfloor + \left\lfloor \frac{2d-m-1}{2m}  \right\rfloor < 0.
\]
This is true because
$$
 \left\lfloor \frac{2m-d+1}{2m} \right \rfloor =0,\ 
 \left\lfloor \frac{2d-m-1}{2m}  \right\rfloor = -1,
$$
where the last equality follows from the assumption $d\le m/2$.
\smallskip

\textbf{Case 3}: $2\le e \le d-1$.
\smallskip

Choose $j=0,i=\ell=2$. Once again, we only need to verify \eqref{eq_ineq3}, which becomes\[
\left\lfloor \frac{4m-2d+2e}{2m} \right \rfloor + \left\lfloor \frac{2-2e-2m}{2m}  \right\rfloor < 0.
\]
This is true because
$$
\left\lfloor \frac{4m-2d+2e}{2m} \right \rfloor =1,\
\left\lfloor \frac{2-2e-2m}{2m}  \right\rfloor =-2.
$$

\textbf{Case 4}: $e =d$.
\smallskip

We have to show that the system \eqref{eq_ineq10}--\eqref{eq_ineq2} has no solution. 
\smallskip

First, we show that any solution $(i,j,\ell)$ of the system \eqref{eq_ineq10}--\eqref{eq_ineq2} must satisfy $j=0$.
\par

Note that the sum of two terms inside the integral parts of \eqref{eq_ineq3} are less than $1$. Then
$$(2m-d+1)\ell+2e-i-j +2di-2e-(m+2d-1)\ell-(m-1)j  < 2m.$$
Hence
\begin{equation}
\label{eq_ineq3a}
mj+2m-1 \ge (2d-1)i+(m-3d+2)\ell. 
\end{equation}

If $m < 3d-2$, using the condition $\ell \le i$ we get
\[
mj+2m-1\ge (2d-1)i+(m-3d+2)i \ge (m-d+1)i.
\]
By \eqref{eq_ineq2} we have $i \ge (m+2d-2)j+1$. Therefore,
\[
mj+2m-1 \ge (m-d+1)(m+2d-2)j+(m-d+1).
\]
If $j\ge 1$, this implies
\[
2d^2-3d \ge m^2+m(d-2).
\]
Since $m\ge 2d$, we get
\[
2d^2-3d \ge m^2+m(d-2) \ge 4d^2+2d(d-2)=6d^2-4d.
\]
From this it follows that $d \ge 4d^2$, which gives a contradiction. Therefore, $j=0$ in this case.

If $m\ge 3d-2$, using \eqref{eq_ineq3a} we get
\begin{align*}
mj+2m-1  &\ge (2d-1)((m+2d-2)j+1)+(m-3d+2)\ell\\ 
         &\ge (2d-1)((m+2d-2)j+1).
\end{align*}
This implies
\[
2(m-d) \ge (2d-2)(m+2d-1)j.
\]
On the other hand, we have
\[
(2d-2)(m+2d-1)- 2(m-d) = (2d-4)m+(2d-2)(2d-1)+2d > 0.
\]
Hence, 
\[
(2d-2)(m+2d-1) >  2(m-d) \ge (2d-2)(m+2d-1)j.
\]
From this it follows that $j=0$.
\smallskip

Next, we show that if $(i,j,\ell)$ is a solution of the system \eqref{eq_ineq10}--\eqref{eq_ineq2} with $j =0$, then $\ell \le 3$. 
\smallskip

If $j =0$, \eqref{eq_ineq3a} becomes
\begin{equation}
\label{eq_ineq4}
2m-1 \ge (2d-1)i+(m-3d+2)\ell. 
\end{equation}
As $i\ge \ell$, this implies
\[
2m-1 \ge (2d-1)\ell + (m-3d+2)\ell = (m-d+1)\ell.
\]
If $\ell\ge 4$, we have $2m-1\ge 4(m-d+1)$, which yields $2m\le 4d-5$, a contradiction. Therefore, $\ell\le 3$.\par

Now, we may assume that $j = 0$  and $\ell \le 3$.
\smallskip 

\textbf{Case 4a}: $\ell=0$.
\smallskip 

Then \eqref{eq_ineq3} becomes
\begin{equation}
\label{eq_ineq3_l=0}
\left\lfloor \frac{2d-i}{2m} \right \rfloor + \left\lfloor \frac{2d(i-1)}{2m}  \right\rfloor < 0.
\end{equation}
Since $i\ge 1$ by \eqref{eq_ineq2}, this implies
\[
\left\lfloor \frac{2d-i}{2m} \right \rfloor  \le -1.
\]
Hence, $i > 2d$. If $i\ge 2d+2$, using \eqref{eq_ineq4} we get
\[
 2m-1 \ge (2d-1)i \ge (2d-1)(2d+2)=4d^2+2d-2
\]
Since $d \ge 2$, this implies $m > 2d^2$, a contradiction to the assumption $\sqrt{m/2} \le d$.
Therefore, we must have $i=2d+1$. Hence,  
\begin{align*}
\left\lfloor \frac{2d-i}{2m} \right \rfloor &=\left\lfloor \frac{-1}{2m} \right \rfloor =-1,\\
\left\lfloor \frac{2d(i-1)}{2m}  \right\rfloor &=\left\lfloor \frac{4d^2}{2m}  \right\rfloor \ge 1.
\end{align*}
This shows that \eqref{eq_ineq3_l=0} is not fulfilled.
\smallskip 

\textbf{Case 4b}:  $\ell=1$.
\smallskip 

Then \eqref{eq_ineq3} becomes
\begin{equation}
\label{eq_ineq3_l=1}
\left\lfloor \frac{2m+d+1-i}{2m} \right \rfloor + \left\lfloor \frac{2d(i-2)-m+1}{2m}  \right\rfloor < 0.
\end{equation}
Since $i \ge 1$ and  $m \ge 2d$, 
\[
(2d(i-2)-m+1)+2m\ge m+1-2d>0.
\]
This implies
\[
\frac{2d(i-2)-m+1}{2m} \ge -1.
\]
Hence,  from \eqref{eq_ineq3_l=1} we get
\[
\frac{2m+d+1-i}{2m} <1.
\]
Consequently, $i > d+1$. 
Together with \eqref{eq_ineq4}, it implies
\[
2m-1 \ge (2d-1)(d+2)+m-3d+2.
\]
This yields $m\ge 2d^2+1$, a contradiction to the assumption $\sqrt{m/2} \le d$.
\smallskip 

\textbf{Case 4c}: $\ell=2$.
\smallskip 

Then \eqref{eq_ineq3} becomes
\[
\left\lfloor \frac{2(2m-d+1)+2d-i}{2m} \right \rfloor + \left\lfloor \frac{2di-2d-2(m+2d-1)}{2m}  \right\rfloor < 0.
\]
Hence
\begin{equation}
\label{eq_ineq3_l=2}
\left\lfloor \frac{4m+2-i}{2m} \right \rfloor + \left\lfloor \frac{2d(i-3)-2m+2}{2m}  \right\rfloor < 0.
\end{equation}
From \eqref{eq_ineq4} we get
$$2m-1 \ge (2d-1)i+2(m-3d+2).$$
Hence $(2d-1)i\le 6d-5$. This implies $i <3$. Since $i \ge \ell =2$, we get $i=2$. 
In this case,
$$\left\lfloor \frac{4m+2-i}{2m} \right \rfloor + \left\lfloor \frac{2d(i-3)-2m+2}{2m}  \right\rfloor = \left\lfloor \frac{2m - 2d +2}{2m}  \right\rfloor = 0,$$
which shows that \eqref{eq_ineq3_l=2} is not fulfilled.
\smallskip 

\textbf{Case 4d}: $\ell=3$.
\smallskip 

Then $i\ge 3$ because of \eqref{eq_ineq10}. If $i\ge 4$, from \eqref{eq_ineq4} we get
\[
2m-1 \ge 4(2d-1)+3(m-3d+2).
\]
This yields $m\le d-3$, a contradiction. Therefore, $i=3$. 
Now, \eqref{eq_ineq3} becomes
\[
\left\lfloor \frac{3(2m-d+1)+2d-3}{2m} \right \rfloor + \left\lfloor \frac{6d-2d-3(m+2d-1)}{2m}  \right\rfloor < 0.
\]
Hence,
\[
\left\lfloor \frac{6m-d}{2m} \right \rfloor + \left\lfloor \frac{3-2d-3m}{2m}  \right\rfloor < 0.
\]
This inequality does not hold because
\[
\left\lfloor \frac{6m-d}{2m} \right \rfloor =2,\
\left\lfloor \frac{3-2d-3m}{2m}  \right\rfloor =-2.
\]

So we have seen that the system \eqref{eq_ineq10}--\eqref{eq_ineq2} has no solution if $e=d$. 
This concludes the proof of Lemma \ref{thm_periodic_large_d}.
\end{proof}

\begin{proof}[Proof of Theorem \ref{periodic}]
If $d=0$, the conclusion follows from Lemma \ref{thm_periodic_d=m}.  \par
If $d \le m/2$, we set $m_1= cm$ and $d_1= cd$, where $c =\max\{{\lceil m/(2d^2)\rceil},2\}$. 
It is easy to see that $m_1 \ge  4$, $d_1\ge 2$ and $\sqrt{m_1/2} \le d_1 \le m_1/2$. Consider the ideals
\begin{align*}
M_1 &=(x^{2m_1+1-d_1},y^{m_1+2d_1-1},z^{m_1+2d_1-1},xy^{m_1-1}z)^2,\\
P_1 &=(x^{2m_1},y^{2m_1}),\\
Q_1 &=(z^2).
\end{align*}
By Lemma \ref{thm_periodic_large_d}, $M_1^t \subseteq P_1^t+Q_1^t$ if and only if $t \equiv d_1$ modulo $m_1$. Let 
$M=M_1^c,\ P=P_1^c,\ Q=Q_1^c.$
Then $M^t \subseteq P^t+Q^t$ if and only if $ct \equiv d_1$ modulo $m_1$, which is satisfied if and only if $t \equiv d$ modulo $m$. \par

If $m/2 < d \le m-1$, then $1\le m-d <m/2$.
As above, we can construct ideals $M_2,P_2,Q_2$ such that $M_2^t \subseteq P_2^t+Q_2^t$ if and only if $t \equiv m-d$ modulo $m$. Let $M=M_2^{m-1}, P=P_2^{m-1},$ $Q=Q_2^{m-1}$. Then $M^t \subseteq P^t+Q^t$ if and only if $(m-1)t \equiv m-d$ modulo $m$, which is satisfied if and only if $t \equiv d$ modulo $m$. \par

Now we only need to set $I = (M) \cap (P,u) \cap (Q,v)$ in both cases.
By Proposition \ref{lem_key}, we have 
$$
\depth R/I^{(t)}= \begin{cases}
2 &\text{if $t \equiv d$ modulo $m$},\\
1 &\text{otherwise}.
\end{cases}
$$
The proof of Theorem \ref{periodic} is now complete.
\end{proof}

%%%%%%%%%%%%%%%%%%%%%%%%%%

\section{Manipulation of symbolic depth functions}

In this section we present techniques which allow us to obtain new symbolic depth functions from 
existing symbolic depth functions.

Let $A$ and $B$ be polynomial rings over a field $k$.
Let  $I\subseteq A$ and $J \subseteq B$ be non-zero proper homogeneous ideals. 
Let $R = A\otimes_k B$.
For simplicity, we use the same symbols $I$ and $J$ to denote the usual extensions of $I$ and $J$ in $R$ if we are working in the algebra $R$. 
 
Moreover, we call the {\em unmixed part} of an ideal 
the intersection of the primary components associated to its minimal primes.
By definition, the $t$-th symbolic power is just the unmixed part of the $t$-th ordinary power.

\begin{prop} \label{product}
$\depth R/(IJ)^{(t)} = \depth A/I^{(t)}+\depth B/J^{(t)} +1.$
\end{prop}

\begin{proof}
By \cite[Lemma 1.1]{HT}, we have the formula $IJ = I \cap J$.
Applying this formula to the ideals $I^t$ and $J^t$, we get
$(IJ)^t = I^tJ^t = I^t \cap J^t$. 
Hence, the unmixed part of $(IJ)^t$ is the intersection of the unmixed parts of $I^t$ and $J^t$. 
Since the unmixed parts of $I^t$ and $J^t$ in $R$ are just the extension of those in $A$ and $B$, 
$(IJ)^{(t)} = I^{(t)} \cap J^{(t)} =  I^{(t)}J^{(t)}$.
By \cite[Lemma 3.2]{HT} we have 
$$\depth R/I^{(t)}J^{(t)} = \depth A/I^{(t)}+\depth B/J^{(t)} +1.$$
\end{proof}

Using Proposition \ref{product} we can add up symbolic depth functions to obtain new symbolic depth functions. 
However, the new symbolic depth functions have higher values.
For instance, if $I$ and $J$ are relevant ideals, then $\depth A/I^{(t)} \ge 1$ and $\depth B/J^{(t)} \ge 1$, hence
$\depth R/(IJ)^{(t)} \ge 3.$ The values of the new symbolic depth functions will be even higher if we add up several symbolic depth functions. To get symbolic depth functions with lower values, we need a technique to reduce the depth of symbolic powers.
For that we have to find a Bertini-type theorem in the following sense.

Let $R = k[x_1,...,x_n]$ be a polynomial ring over $k$.
Let $I$ be a homogeneous ideal.
We need to find a linear form $f \in R$ such that for all $t \ge 1$, $f$ is a non-zerodivisor on $I^{(t)}$ and if we set
$S = R/(f)$ and $Q = (I,f)/(f)$, then
$$S/Q^{(t)} = R/(I^{(t)},f).$$

There is an obstacle for such a theorem, namely that $f$ has to be the same element for all symbolic power $I^{(t)}$, which is an infinite family of ideals. 
Such a theorem can be found by using the following construction.

We will replace $R$ by the polynomial ring $R(u) := R \otimes_k k(u)$, where 
$k(u) = k(u_1,...,u_n)$ is a purely transcendental extension of $k$.
Set  
$$f_u := u_1x_1 + \cdots + u_nx_n.$$
We call $f_u$ a {\em generic linear form}.

First, we have to study the unmixed part of the ideal $(I,f_u)$ for an ideal $I$ of $R(u)$.

\begin{lem} \label{transfer}
Assume that $\dim R/I \ge 2$. Then \par
{\rm (i)} The set of the minimal primes of $(I,f_u)$ is the union of the sets of the minimal primes $P$ of $(\wp,f_u)$ with $P \cap R = \wp$, where $\wp$ is a minimal prime of $I$. \par
{\rm (ii)} $(I,f_u)$ is an unmixed ideal if $I$ is an unmixed ideal and $\depth R/I \ge 2$.
\end{lem}

\begin{proof} 
Let $R[u] = R[u_1,...,u_n]$. Since $R(u)$ is a localization of $R[u]$, 
it suffices to prove (i) and (ii) for ideals in $R[u]$.
Replacing $R$ by the quotient ring $R/I$ we may assume that $I = 0$,
where $R$ is now a standard graded algebra over $k$ with the maximal homogeneous ideal $(x_1,...,x_n)$.
Then we have to prove the following statements under the assumption $\dim R \ge 2$:
 \par
{\rm (i')} The set of the minimal primes of $(f_u)$ is the union of the sets of the minimal primes $P$ of $(\wp,f_u)$ with $P \cap R = \wp$,
where $\wp$ is a minimal prime of $R$. \par
{\rm (ii')} $(f_u)$ is an unmixed ideal if $R$ is unmixed with $\depth R \ge 2$.

The transfer of properties between $R$ und $R[u]/(f_u)$ was already studied in a more general setting in \cite{Tr}. 
In fact, (i') is a consequence of \cite[Lemma 1.5]{Tr}\footnote{There is a typo in \cite[Lemma 1.5]{Tr}. In the formula for $M_1$ one has to replace $\text{Assm}(R)\setminus {\mathcal V}(I_F)$
by $\text{Assm}(R)\cap {\mathcal V}(I_F)$.}
Since a ring is unmixed if and only if it satisfies Serre's condition $S_1$, (ii') follows from
\cite[Theorem 3.1]{Tr}. We leave the reader to check the details.  
\end{proof}

\begin{prop} \label{Bertini}
Let $I$ be an ideal with $\depth R/I^{(t)} \ge 2$ for some $t \ge 1$.
Let $S = R(u)/(f_u)$ and $Q = (I,f_u)/(f_u)$.
Then $f_u$ is a regular element on $I^{(t)}R(u)$  and 
$$S/Q^{(t)} = R(u)/(I^{(t)},f_u).$$
\end{prop}

\begin{proof} 
All associated primes of $I^{(t)}R(u)$ are of the form $\wp R(u)$, where $\wp$ is a minimal primes of $I^t$ in $R$.
Since $\dim R/I^{(t)} \ge \depth R/I^{(t)} \ge 2$, $\wp \neq (x_1,...,x_n)$. Therefore, $f_u \not\in \wp R(u)$.
From this it follows that $f_u$ is a regular element on $I^{(t)}R(u)$.

It is clear that $S/Q^{(t)} = R(u)/(I^t,f_u)^{(1)}$. 
To prove that $S/Q^{(t)} = R(u)/(I^{(t)},f_u)$ we have to prove that 
$(I^t,f_u)^{(1)} = (I^{(t)},f_u).$

Since $I^t$ and $I^{(t)}$ share the same minimal primes, 
so do $(I^t,f_u)$ and $(I^{(t)},f_u)$ by Lemma \ref{transfer}(i). 
This implies $(I^t,f_u)^{(1)} \subseteq (I^{(t)},f_u)^{(1)}$.
By Lemma \ref{transfer}(ii), $(I^{(t)},f_u)^{(1)} = (I^{(t)},f_u)$.
Therefore, $(I^t,f_u)^{(1)} \subseteq (I^{(t)},f_u)$.
It remains to show that $(I^{(t)},f_u) \subseteq (I^t,f_u)^{(1)}$.
For this it suffices to show that $I^{(t)} \subseteq (I^t,f_u)^{(1)}$.
That will be done if we can show that every primary component 
associated with a minimal prime of $(I^t,f_u)$ contains $I^{(t)}$.

Let $P$ be an arbitrary minimal prime of $(I^t,f_u)$.
Then $(I^t,f_u)R(u)_P \cap R(u)$ is the $P$-primary component of $(I^t,f_u)$.
By Lemma \ref{transfer}(i), $P$ is a minimal prime of $(\wp,f_u)$ for some minimal prime $\wp$ of $I^t$ with $P \cap R = \wp$. Since $R \setminus \wp \subset R(u) \setminus P$, $I^tR(u)_P$ is a localization of $I^tR_\wp(u)$, where $R_\wp(u) := R_\wp \otimes_kk(u)$.
Since $I^tR_\wp$ is a primary ideal, so is $I^tR_\wp(u)$.
Hence, $I^tR(u)_P$ is also a primary ideal.
Since $I^{(t)}R(u)_P$ is the unmixed part of $I^tR(u)_P$, this implies $I^{(t)}R(u)_P =I^tR(u)_P$.
Therefore, $I^{(t)} \subseteq I^tR(u)_P \cap R \subseteq (I^t,f_u)R(u)_P \cap R$, as required.
\end{proof}

\begin{cor} \label{lower}
Let $\phi(t)$ be a symbolic depth function over $k$ such that $\phi(t) \ge 2$ for all $t \ge 1$.
Then $\phi(t) - 1$ is also a symbolic depth function over a purely transcendental extension of $k$.
\end{cor}

\begin{proof}
Let $R$ be a polynomial ring over $k$ and $I$ a homogeneous ideal in $R$ 
such that $\depth R/I^{(t)} = \phi(t)$ for $t \ge 1$. 
Let $S = R(u)/(f_u)$ and $Q = (I,f_u)/(f_u)$, where $f_u$ is a generic linear form.
By Proposition \ref{Bertini} we have
$$\depth S/Q^{(t)} = \depth R/I^{(t)}-1$$
for all $t \ge 1$.
\end{proof}

Now we are able to lower all values of a sum of two symbolic depth functions by one and still get a symbolic depth function.

\begin{cor} \label{additive}
Let $\phi(t)$ and $\psi(t)$ be two symbolic depth functions over a field $k$.
Then $\phi(t) + \psi(t) - 1$ is a symbolic depth function over a purely transcendental extension of $k$.
\end{cor}

\begin{proof}
Let $A, B$ be two polynomial rings over $k$ and $I \subset A$, $J \subset B$ two homogeneous ideals such that
$\depth A/I^{(t)} = \phi(t)$ and $\depth B/J^{(t)} = \psi(t)$ for $t \ge 1$.
Let $R = A\otimes_k B$. By Proposition \ref{product}, we have
$$\depth R/(IJ)^{(t)} = \phi(t) + \psi(t) + 1$$
for $t \ge 1$. Hence, $\phi(t) + \psi(t) + 1$ is a symbolic depth function over $k$.
Since $\phi(t) \ge 1$ and $\psi(t) \ge 1$, $\depth R/(IJ)^{(t)} \ge 3$ for all $t \ge 1$.
Therefore, we only need to apply Corollary \ref{lower} twice in order to see that  $\phi(t) + \psi(t) - 1$ is a symbolic depth function over a purely transcendental extension of $k$. 
\end{proof}

The symbolic depth function $\phi(t) + \psi(t) - 1$ is the best possible we can get from $\phi(t)$ and $\psi(t)$ by the above method. Namely, $\phi(t) + \psi(t) - 2$ is not always a symbolic depth function.
If there exists $t$ such that $\phi(t) = \psi(t) = 1$, then $\phi(t) + \psi(t) - 2 = 0$ and 0 can not be the depth of  $R/I^{(t)}$ for a relevant ideal $I$.
\par

One may ask whether in the above lemmas, $\phi(t)-1$ and $\phi(t) + \psi(t) - 1$ are symbolic depth functions over $k$. 
We shall see that this question has a positive answer if $k$ is an uncountable field. For that it suffices to prove a Bertini-type theorem like Proposition \ref{Bertini} without extending $k$. 

Let $J$ be an ideal in $R(u)$. For any $\a \in k^n$, we define 
$$J_\a := \{f(\a)|\ f(u) \in J \cap R[u]\}.$$
Obviously, $J_\a$ is an ideal in $R$. We call $J_\a$ the {\it specialization} of $J$ with respect to the substitution $u \to \a$. This notion was first studied by W. Krull and A. Seidenberg (see \cite{Se}). It can be generalized to define a specialization $M_\a$ of a finitely generated module $M$ over $R(u)$ that preserves many properties of $M$ for almost all $\a$, i.e. for all $\a$ in a non-empty Zariski-open subset of $k^n$. We refer the reader to \cite{NT} for details.\footnote{The proof for Proposition 3.2(ii) and (iii) of \cite{NT} has errors, which can be corrected as follows.  For Proposition 3.2(ii), we consider the exact consequence $0 \to M \cap N \to L \to (L/M) \oplus (L/N)$ and apply Corollary 2.5(i) and Lemma 3.1. Proposition 3.2(iii) follows from the exact sequence $M \oplus N \to L \to L/(M+N) \to 0$ by applying Corollary 2.5(ii) and Lemma 3.1.}

\begin{lem} \label{non-zerodivisor}
Let $k$ be an infinite field. Let $J$ be an ideal in $R(u)$ and $f \in k[u,X]$ a regular element on $J$. 
For almost all $\a$, $f(\a,X)$ is a non-zerodivisor on $J_\a$.
\end{lem}

\begin{proof}
We have $J: f = J$. By \cite[Proposition 3.2(i) and Proposition 3.6]{NT}, for almost all $\a$,
$$0 = (J:f/J)_\a = (J:f)_\a/J_\a = (J_\a: f(\a,X))/J_\a,$$ 
which implies $J_\a: f(\a,X) = J_\a$. 
\end{proof}

\begin{lem} \label{unmixed2}
Let $k$ be an infinite field. Let $J$ be an ideal in $R(u)$ and $U$ the unmixed part of $J$. 
For almost all $\a$, $U_\a$ is the unmixed part of $J_\a$.
\end{lem}

\begin{proof}
Let $J = Q_1 \cap \cdots \cap Q_r$ be a primary decomposition of $J$. Let $P_i = \sqrt{Q_i}$, $i = 1,...,r$.
Suppose that $P_1,...,P_m$, $m \le r$, are the minimal associated primes of $J$. Then
$U = Q_1 \cap \cdots \cap Q_m$. By \cite[Proposition 3.2(ii)]{NT}, $J_\a = (Q_1)_\a \cap \cdots \cap (Q_r)_\a$ and $U_\a = (Q_1)_\a \cap \cdots \cap (Q_m)_\a$
for almost all $\a$.  

For $i = 1,...,m$, every associated primes of $(P_i)_\a$ and $(Q_i)_\a$ have the same height like $P_i$ and $Q_i$ for almost all $\a$ \cite[Appendix, Theorem 6]{Se}.  
Since $(Q_i)_\a \subseteq (P_i)_\a \subseteq \sqrt{(Q_i)_\a}$, we have
$\sqrt{(Q_i)_\a} = \sqrt{(P_i)_\a}$.
Therefore, $(Q_i)_\a$ and $(P_i)_\a$ share the same associated primes.

For $i = m+1,...,r$, there exists $j \le m$ such that $P_i \supset Q_j$ with $\height P_i > \height Q_j$.
This implies $\height (P_i)_\a  >  \height (Q_j)_\a$ for almost all $\a$.
Since $(P_i)_\a \supset (Q_j)_\a$, every associated prime of $(P_i)_\a$ properly contains a minimal associated prime of $(Q_j)_\a$.
From this it follows that the associated primes of $(Q_i)_\a$ are non-minimal associated primes of $J_\a$ for $i = m+1,...,r$.
Thus, the minimal associated primes of $(Q_1)_\a,...,(Q_m)_\a$ are precisely the minimal associated primes of $J_\a$. Hence, $U_\a$ is the unmixed part of $J_\a$ for almost all $\a$.
\end{proof}

For $\a = (a_1,...,a_n) \in k^n$, we set $f_\a := a_1x_1 + \cdots + a_nx_n$.

\begin{prop} \label{uncountable}
Let $k$ be an uncountable field. Let $I$ be a homogeneous ideal in $R$ with $\depth R/I^{(t)} \ge 2$ for all $t \ge 1$.
Let $S = R/(f_\a)$ and $Q_\a = (I,f_\a)/(f_\a)$. Then there is $\a \in k^n$ such that for all $t \ge 1$, 
$f_\a$ is a regular element on $I^{(t)}$ and 
$$S/(Q_\a)^{(t)} = R/(I^{(t)},f_\a).$$
\end{prop}

\begin{proof}
By Proposition \ref{Bertini}, $f_u$ is a regular element on $I^{(t)}$ and $(I^{(t)},f_u)$ is the unmixed part of $(I^t,f_u)$. 
It is easy to check that $(I^t,f_\a) = (I^t,f_u)_\a$ for almost all $\a$.
By Lemma \ref{non-zerodivisor} and Lemma \ref{unmixed2}, there exists a non-empty Zariski-open set $U_t$ in $k^n$ such that $f_\a$ is a regular element on $I^{(t)}$ and $(I^{(t)},f_\a)$ is the unmixed part of $(I^t,f_\a)$ for all $\a \in U_t$.
Since $k$ is an uncountable field, $\bigcap_{t\ge 1} U_t$ is uncountable (see e.g. \cite[Lemma 3.1]{Nh}).
Therefore, there is $\a \in k^n$ such that for all $t \ge 1$, $f_\a$ is a regular element on $I^{(t)}R$ and 
$S/(Q_\a)^{(t)} = R(u)/(I^{(t)},f_\a).$ 
\end{proof} 

With regard to Proposition \ref{uncountable} we raise the following question.

\begin{quest} 
Let $R$ be a polynomial over an infinite field and $I$ a homogeneous ideal in $R$ with $\depth R/I^{(t)} \ge 2$ for all $t \ge 1$. Does there exist a linear form $f \in R$ such that $f$ is a regular element on $I^{(t)}$ and
if we set $S = R/(f)$ and $Q = (I,f)/(f)$, then
$$S/Q^{(t)} = R/(I^{(t)},f)$$
for all $t \ge 1$?
\end{quest}

This question has a negative answer if the base field is finite.

\begin{ex}
Let $k$ be a finite field and $R = k[x_1,...,x_n]$, $n \ge 3$. Then $R$ has only finitely many linear forms.
Let $I$ be the principal ideal generated by the product of all linear forms of $R$. Then $I^{(t)} = I^t$ and
$\depth R/I^t = n-1 \ge 2$ for all $t \ge 1$. It is clear that any linear form $f \in R$ is not a zerodivisor on $I^t$ for all $t \ge 1$.
\end{ex}

%%%%%%%%%%%%%%%%%%%%%%%%%%%%%%%%%%

\section{Ubiquity of asymptotically periodic symbolic depth functions}

In this section we prove the following result, which is the main contribution of this paper.

\begin{thm} \label{ubiquity}
Let $\phi(t)$ be any asymptotically periodic positive numerical function. Given a field $k$,
there exist a polynomial ring $R$ over a purely transcendental extension of $k$ and a homogeneous ideal $I \subset R$ in such that $\depth R/I^{(t)}=\phi(t)$ for $t\ge 1$.
\end{thm}

The idea is to build up any asymptotically periodic positive numerical function from basic symbolic depth functions by using the operations 
\begin{align*}
\overline \phi(t) & := \phi(t)-1,\\
(\phi\star\psi)(t) & := \phi(t)+\psi(t)-1
\end{align*}
By Corollary \ref{lower} and Corollary \ref{additive}, if $\phi(t)$ is a symbolic depth function with $\phi(t) \ge 2$ for all $t \ge 1$, then $\overline \phi(t)$ is a symbolic depth function, and if $\phi(t)$ and $\psi(t)$ are arbitrary symbolic depth functions, then $(\phi\star\psi)(t)$ is a symbolic depth function.
\par

The basic symbolic depth functions are functions of the following types, whose existence has been shown in Section 4.
\par

\begin{enumerate} 
\item[{\bf A}:] 1,...,1,2,2,... , which is a monotone function converging to 2,
\item[{\bf B}:] 1,...,1,2,1,1,... , which has the value 2 at only one position, 
\item[{\bf C}:] 1,1,1,... or 1,..,1,2,1,..,1,1,..,1,2,1,..,1,... , which is a periodic function with a period of the form 1,..,1,2,1,..,1, where 2 can be at any position.
\end{enumerate}

\begin{lem} \label{simple}
Any asymptotically periodic positive numerical function is obtained from finitely many functions of types {\bf A, B, C} by using the operations $\overline \phi$ with $\phi(t) \ge 2$ for all $t \ge 1$ and $\phi\star\psi$.
\end{lem}

\begin{proof}
First, we note that the map $\phi(t) \to \overline \phi(t)$ gives an one-to-one correspondence between positive numerical functions and non-negative numerical functions with 
$$
\overline {(\phi\star\psi)}(t) = \overline\phi(t)+\overline\psi(t)
$$
Then we have to show that any asymptotically periodic non-negative numerical function is obtained from finitely many 0-1 functions of the following types by using the operation $\bar\phi$ with $\phi(t) \ge 1$ for all $t \ge 1$ and the usual addition of functions:
\begin{enumerate}
\item[{\bf A'}:] 0,...,0,1,1,... , which is a monotone function converging to 1,
\item[{\bf B'}:] 0,...,0,1,0,0,... , which has the value 1 at only a place,
\item[{\bf C'}:] 0,0,0,... or 0,..,0,1,0,..,0,0,..,0,1,0,..,0,... , which is a periodic function with a period of the form 0,..,0,1,0,..,0, where 1 can be at any position.
\end{enumerate}

Let $\phi(t)$ be an arbitrary asymptotically periodic non-negative numerical function.
It is obvious that $\phi(t)$ is a sum of asymptotically periodic numerical 0-1 functions.
Hence, we may assume that $\phi(t)$ is a 0-1 function.
We may further assume that $\phi(t)$ is not the function 0,0,0,..., i.e. $\phi(t) = 1$ for some $t \ge 1$. 
Let $c$ be the length of a period of $\phi(t)$ for $t \gg 0$. 
Choose $s$ to be a multiple of $c$ such that $\phi(t)$ is periodic for $t > s$ and $\phi(t) = 1$ for some $t \le s$.
Then there exists a unique periodic 0-1 function $\phi_1(t)$ such that $\phi_1(t) = \phi(t)$ for $t > s$. Let
$$
\phi_2(t) = \begin{cases}
0 &\text{if $t \le s$ and $\phi_1(t) = 1$},\\
1 &\text{otherwise}.
\end{cases}
$$
Then 
$$\phi_1(t) + \phi_2(t) = 
\begin{cases}
1 & \text{if $t \le s$},\\
\phi(t)+1 & \text{if $t > s$}.
\end{cases}
$$
Therefore, if we set
$$\phi_3(t) = 
\begin{cases}
\phi(t) &\text{if $t \le s$},\\
0 & \text{if $t > s$},
\end{cases}
$$
we have
$$\phi(t) = \phi_1(t) + \phi_2(t) + \phi_3(t) - 1 = {(\overline{\phi_1+\phi_2}})(t) + \phi_3(t).$$ 
It is clear that $\phi_1(t)$ is a sum of functions of type {\bf C'}. 
Since $\phi_2(t)$ is convergent to 1, $\phi_2(t)$ is a sum of a function of type {\bf A'} with functions of type {\bf B'}. Since $\phi_3(t)$ is convergent to 0 but not constant, $\phi_3(t)$ is a sum of functions of type {\bf B'}.
\end{proof} 

\begin{proof}[Proof of Theorem \ref{ubiquity}]
By Lemma \ref{simple}, any asymptotically periodic positive numerical function can be obtained from functions of types {\bf A, B, C} by the operations $\overline \phi$ with $\phi(t) \ge 2$ for all $t \ge 1$ and $\phi\star\psi$. 
By Corollary \ref{lower} and Corollary \ref{additive}, these operations preserve the property of being a symbolic depth function.
Therefore, any asymptotically periodic positive numerical function is the symbolic depth function of a homogeneous ideal.
\end{proof}

By the Auslander-Buchsbaum formula we have
$$\pd I^{(t)} = \dim R - \depth R/I^{(t)} -1.$$
Therefore, one can deduce from Theorem \ref{ubiquity} the following result on the behavior of the function $\pd I^{(t)}$.

\begin{thm} \label{pd}
Let $\psi(t)$ be an arbitrary asymptotically periodic non-negative numerical function and $m = \max_{t \ge 1}\psi(t)$.
Given a field $k$, there is a number $c$ such that there exist a polynomial ring $R$ in $m+c+2$ variables over a purely transcendental extension of $k$ and a
homogeneous ideal $I \subset R$ for which $\pd I^{(t)} = \psi(t) + c$ for all $t\ge 1$.
\end{thm}

\begin{proof}
Set $\phi(t) = m - \psi(t) + 1$ for all $t \ge 1$. Then $\phi(t)$ is an asymptotically periodic positive numerical function.
By Theorem \ref{ubiquity}, there exist a polynomial ring $R$ over a purely transcendental extension of $k$ and a
homogeneous ideal $I \subset R$ such that $\depth R/I^{(t)} = \phi(t)$ for all $t \ge 1$.
Let $n$ be the number of variables of $R$. Set $c = n-m-2$.
Then 
$$\pd I^{(t)} = n - \phi(t) - 1 = n - m + \psi(t) - 2 = \psi(t) + c$$
for all $t \ge 1$. 
\end{proof}

Due to the use of reductions by generic linear forms in Corollary \ref{lower} and Corollary \ref{additive}, the constructed ideal with a given symbolic depth function is a non-monomial ideal in a polynomial ring over a purely transcendental extension of $k$. Using Proposition \ref{uncountable}, we can construct such an ideal in a polynomial ring over any uncountable field.  This leads us to the following question to which we could not give any answer.

\begin{quest} \label{field}
Given any asymptotically periodic positive numerical function $\phi(t)$, 
do there exist a polynomial ring $R$ over {\em any field} and a {\em monomial ideal} $I \subset R$ such that $\depth R/I^{(t)} = \phi(t)$ for all $t \ge 1$?
\end{quest}

Note that the analogous question for the depth function of the ordinary powers of a homogeneous ideal has a positive answer \cite[Theorem 4.1]{HNTT}. \par

Another issue is the smallest number $n$ of variables of a polynomial ring $R$ which contains a homogeneous ideal $I$ with a given symbolic depth function. This number determines the smallest number $c$ in Theorem \ref{pd}.
The proof of Theorem \ref{ubiquity} uses a high number of variables compared to the values of $\depth R/I^{(t)}$.
However, for all constructed symbolic depth functions of types {\bf A, B, C} (except 1,1,1,..), we have $n = 5$ and $\height I = 2$. 
Inspired by this fact we raise the following question.

\begin{quest} \label{variables}
Let $\phi(t)$ be an asymptotically periodic positive numerical function and $m = \max_{t\ge1} \phi(t)$.
Does there exist a polynomial ring $R$ in $m+3$ variables that contains a height 2 homogeneous ideal $I$
such that $\depth R/I^{(t)} = \phi(t)$ for all $t \ge 1$?
\end{quest}

\begin{ex} 
Let $\phi(t)$ be the numerical function 2,1,2,2,.... Then $\phi(t) = \phi_1(t) + \phi_2(t) - 1,$
where $\phi_1(t)$ and $\phi_2(t)$ are the functions 1,1,2,2,... and 2,1,1,1,....  By the proof of Lemma \ref{simple} and Theorem \ref{ubiquity}, we can construct a height 2 homogeneous ideal in 8 variables having the symbolic depth function $\phi(t)$. Now we are going to construct a height 2 homogeneous ideal in 5 variables having the same symbolic depth function $\phi(t)$. \par

Let $R=k[x,y,z,u,v]$ and $I=M\cap (P,u)\cap (Q,v)$, where
\begin{align*}
M &= (x^7,y^7,x^2y^2z,z^5)^2,\\
P &= (x^7,y^7),\\
Q &= (z^2).
\end{align*}
We claim that 
\[
\depth R/I^{(t)}= \begin{cases}
                  1, &\text{if $t=2$},\\
                  2, &\text{otherwise}.
                 \end{cases}
\]
By Proposition 4.1, we have to show that $M^t\subseteq P^t+Q^t$ if and only if $t\neq 2$.

It is clear that $M\subseteq P+Q$. Since
$$
x^{13}y^6z^3 = x^7(x^2y^2z)^3 \in (x^7,y^7,x^2y^2z,z^5)^4 = M^2,
$$
$$x^{13}y^6z^3\notin (x^7,y^7)^2+(z^4)=P^2+Q^2,$$
we have $M^2\not\subseteq P^2+Q^2$.  

For $t \ge 3$, we first note that a minimal monomial generator of $M^t$ has the form
\[
f=(x^7)^i(y^7)^j(x^2y^2z)^{\ell}(z^5)^{2t-i-j-\ell}=x^{7i+2\ell}y^{7j+2\ell}z^{10t-5(i+j)-4\ell}
\]
where $0\le i,j,\ell$ and $i+j+\ell \le 2t$.
Therefore, $M^t\subseteq P^t+Q^t=(x^7,y^7)^t+(z^{2t})$ if and only if the following system has no solution:
\begin{align}
0\le i,j,\ell, i+j+\ell &\le 2t, \label{eq_63_cond1}\\
\left\lfloor \frac{7i+2\ell}{7} \right\rfloor + \left\lfloor \frac{7j+2\ell}{7} \right\rfloor &< t, \label{eq_63_cond2}\\
10t-5(i+j)-4\ell & < 2t. \label{eq_63_cond3}
\end{align}
Assume that this system has a solution for some $t\ge 3$. Then the sum of the two fractions inside the integral parts of \eqref{eq_63_cond2} must be less than $t+1$. Hence,
$
7(i+j)+4\ell < 7(t+1), 
$
which implies
\begin{equation}
\label{eq_63_cond4}
4\ell \le 7t+6-7(i+j).
\end{equation}
From \eqref{eq_63_cond3} we get
\begin{equation}
\label{eq_63_cond5}
8t \le 5(i+j)+4\ell -1.
\end{equation}
Combining \eqref{eq_63_cond4} and \eqref{eq_63_cond5}, we see that
\[
8t\le 5(i+j)-1+7t+6-7(i+j),
\]
which implies
\begin{equation}
\label{eq_63_cond6}
2(i+j)+t \le 5.
\end{equation}

If $i+j=0$,  from \eqref{eq_63_cond5} we get $8t\le 4\ell-1$, so $2t<\ell$. This contradicts \eqref{eq_63_cond1}.

If $i+j\ge 1$, then \eqref{eq_63_cond6} forces $i+j=1$ and $t=3$. From \eqref{eq_63_cond1}  we get $\ell \le 5$. From \eqref{eq_63_cond5} we get $\ell \ge 5$. Hence, $\ell=5$. Since $i+j=1$, we have $\{i,j\}=\{0,1\}$. Consequently, \eqref{eq_63_cond2} implies
\[
3=\left\lfloor \frac{17}{7} \right \rfloor + \left \lfloor \frac{10}{7} \right \rfloor < 3.
\]
This is a contradiction. Hence, $M^t \subseteq P^t+Q^t$ for $t\ge 3$, as required. \par

As a consequence, we have 
\[
\pd I^{(t)}= \begin{cases}
                  3, &\text{if $t=2$},\\
                  2, &\text{otherwise}.
                 \end{cases}
\]
Therefore, if $\psi(t)$ is the numerical function $2,3,2,2,...$, we can choose $c = 0$ in Theorem \ref{pd}, while its proof only yields $c = 3$.
\end{ex}
 
Question \ref{variables} has probably a negative answer for $m+2$ variables. For, if $\dim R = m+2$ then $I^{(t)}$ is a Cohen-Macaulay ideal for those $t$ for which $\depth R/I^{(t)} = m$. By Hilbert-Burch structure theorem, $I^{(t)}$ is  generated by the maximal minors of an $r \times (r+1)$ matrix for some $r \ge 1$. That would affect the structure of other symbolic powers of $I$. For this reason, the symbolic depth function of $I$ might not be arbitrary. 
 
\begin{rem} 
Suppose that Question \ref{variables} has a positive answer.
Then we can choose $c = 2-\min_{t \ge 1}\psi(t)$ in the proof of Theorem \ref{pd}. 
Therefore, given any asymptotically periodic non-negative numerical function $\psi(t)$,
there exist a polynomial ring $R$ and a height 2 homogeneous ideal $I \subset R$ such that 
$$\pd I^{(t)} = \psi(t) -\min_{t \ge 1}\psi(t)+2$$
for all $t \ge 1$.
\end{rem}

Finally, we would like to raise the following problem.

\begin{quest} 
Does there exist a homogeneous ideal whose symbolic depth function is not asymptotically periodic?
\end{quest}

By Proposition \ref{Rees}, the symbolic Rees algebra of such an ideal has to be non-noetherian.
To find non-noetherian symbolic Rees algebras is a hard problem, which is related to Hilbert's fourteenth
problem \cite{Ro}.
As far as we know, there are only examples of non-noetherian symbolic Rees algebras for one-dimensional ideals  (see e.g. \cite{Cu, Hu, Ro}). In this case, we have $\depth R/I^{(t)} = 1$ for all $t \ge 1$, which implies that the symbolic depth function is a constant function.

%%%%%%%%%%%%%%%%%%%

\end{document}